\documentclass[11pt]{amsart}
\usepackage[dvipsnames]{xcolor}
\usepackage[utf8]{inputenc}
\usepackage{amssymb,amsmath,amsthm,amsfonts,fullpage,float,latexsym,bbm,microtype,cite}
\usepackage{tikz}
\usetikzlibrary{decorations.pathreplacing}
\usepackage{hyperref}
\hypersetup{colorlinks=true, linkcolor=blue, citecolor=magenta, urlcolor=magenta}

\newtheorem{theorem}{Theorem}[section]
\newtheorem{proposition}[theorem]{Proposition}
\newtheorem{claim}[theorem]{Claim}
\newtheorem{conjecture}[theorem]{Conjecture}
\newtheorem{lemma}[theorem]{Lemma}
\newtheorem{corollary}[theorem]{Corollary}
\theoremstyle{definition}
\newtheorem{remark}[theorem]{Remark}
\newtheorem{definition}[theorem]{Definition}

\newtheorem{example}[theorem]{Example}

\renewcommand{\emptyset}{\varnothing}
\newcommand{\0}{\emptyset}
\DeclareMathOperator{\conv}{conv}
\newcommand{\R}{\mathbb{R}}
\newcommand{\Z}{\mathbb{Z}}
\newcommand{\N}{\mathbb{N}}
\newcommand{\ee}{\mathbf{e}}
\newcommand{\bx}{\mathbf{x}}
\newcommand{\B}{\mathcal{B}}
\newcommand{\F}{\mathcal{F}}
\newcommand{\G}{\mathcal{G}}
\newcommand{\Po}{\mathcal{P}}
\newcommand{\Vol}{\mathrm{Vol}}

\DeclareMathOperator{\ehr}{\mathsf{ehr}}
\DeclareMathOperator{\des}{des}
\DeclareMathOperator{\Des}{Des}
\DeclareMathOperator{\Rel}{Rel} 

\newcommand{\sm}{\setminus}
\newcommand{\defterm}[1]{\textbf{#1}}
\DeclareMathOperator{\Pan}{\mathsf{Pan}} 
\DeclareMathOperator{\rank}{\operatorname{rank}}
\DeclareMathOperator{\CF}{\mathsf{CF}} 

\DeclareMathOperator{\wt}{wt}
\newcommand{\Sym}{\mathfrak{S}}
\newcommand{\elm}[2]{e_{#1}^{#2}} 
\DeclareRobustCommand{\stir}{\genfrac[]{0pt}{}} 
\newcommand{\Gchi}[1]{\chi\big[#1\big]} 
\newcommand{\HH}{\mathcal{H}}

\newcommand\commentout[1]{}

\begin{document}

\title{Ehrhart theory of paving and panhandle matroids}

\author{Derek Hanely}
\address{Department of Mathematics\\
         Penn State Behrend\\
\url{https://behrend.psu.edu/person/derek-hanely}}
\email{derek.hanely@psu.edu}

\author{Jeremy L. Martin}
\thanks{Martin is partially supported by Simons Collaboration Grant \#315347.}
\address{Department of Mathematics\\
         University of Kansas\\
\url{https://jlmartin.ku.edu/}}
\email{jlmartin@ku.edu}

\author{Daniel McGinnis}
\thanks{McGinnis is partially supported by the National Science Foundation under grant DMS-1839918 (RTG)}
\address{Department of Mathematics\\
        Iowa State University\\
\url{https://sites.google.com/iastate.edu/danielmcginnis/home}}
\email{dam1@iastate.edu}

\author{Dane Miyata}
\thanks{Miyata is partially supported by the National Science Foundation under grants DMS-1954050, DMS-2039316 (RTG), and DMS-2053243 (FRG)}
\address{Department of Mathematics\\
         University of Oregon\\
\url{https://math.uoregon.edu/profile/dmiyata}}
\email{dmiyata@uoregon.edu}

\author{George D. Nasr}
\thanks{Nasr was partially supported by the National Science Foundation under  grant DMS-2053243 (FRG)}
\address{Department of Mathematics\\
         Augustana University\\
\url{https://sites.google.com/view/george-d-nasr-math}}
\email{george.nasr@augie.edu}

\author{Andr\'es R. Vindas-Mel\'endez}
\thanks{Vindas-Mel\'endez is partially supported by the National Science Foundation under Award DMS-2102921.}
\address{Department of Mathematics\\
         University of California, Berkeley\\
\url{https://math.berkeley.edu/~vindas}}
\email{andres.vindas@berkeley.edu}

\author{Mei Yin}
\thanks{Yin is partially supported by the University of Denver's Faculty Research Fund 84688-145601.}
\address{Department of Mathematics\\
         University of Denver\\
\url{https://cs.du.edu/~meiyin/}}
\email{mei.yin@du.edu}

\date{\today}


\begin{abstract}
We show that the base polytope $P_M$ of any paving matroid $M$ can be systematically obtained from a hypersimplex by slicing off certain subpolytopes, namely base polytopes of lattice path matroids corresponding to panhandle-shaped Ferrers diagrams.  We calculate the Ehrhart polynomials of these matroids and consequently write down the Ehrhart polynomial of $P_M$, starting with Katzman's formula for the Ehrhart polynomial of a hypersimplex.  
The method builds on and generalizes Ferroni's work on sparse paving matroids.  
Combinatorially, our construction corresponds to constructing a uniform matroid from a paving matroid by iterating the operation of \textit{stressed-hyperplane relaxation} introduced by Ferroni, Nasr, and Vecchi, which generalizes the standard matroid-theoretic notion of circuit-hyperplane relaxation. 
We present evidence that panhandle matroids are Ehrhart positive and describe a conjectured combinatorial formula involving chain forests and Eulerian numbers from which Ehrhart positivity of panhandle matroids will follow.
As an application of the main result, we calculate the Ehrhart polynomials of matroids associated with Steiner systems and finite projective planes, and show that they depend only on their design-theoretic parameters: for example, while projective planes of the same order need not have isomorphic matroids, their base polytopes must be Ehrhart equivalent.
\end{abstract}

\maketitle


\section{Introduction}

The \defterm{Ehrhart function} of a polytope $P\subset \R^d$ is $\ehr_P(t):=|tP\cap \Z^d|$. where $tP=\{t\bx:\ \bx\in P\}$.  Ehrhart theory, developed by Ehrhart in the 1960s (see, e.g., \cite{Ehrhart}) can be regarded as a discrete version of integration: the growth of the Ehrhart function provides information about the volume and surface area of $P$.  Indeed, when $P$ is a lattice polytope (its vertices have integer coordinates), the Ehrhart function is a polynomial in $t$, with degree equal to the dimension of $P$, leading coefficient equal to its normalized volume, second-leading coefficient equal to half the surface area, and constant coefficient 1.
The other coefficients of $\ehr_P(t)$ are more mysterious, and in general can be negative.  While Ehrhart functions can often be calculated by combinatorial means, it is very often easier to describe the Ehrhart function of a polytope than to give an explicit polynomial expression for it.
Ehrhart theory is connected to the combinatorics of simplicial complexes, as well as number theory and discrete analysis; for a comprehensive overview, see~\cite{BeckRobins}.

In this paper, we are exclusively concerned with the Ehrhart theory of matroid base polytopes.  Recall that a \defterm{matroid} $M$ on finite ground set $E$ can be defined by its basis system $\B$, a nonempty collection of subsets of $E$, all of the same size, satisfying a certain \textit{exchange condition} (see Definition~\ref{defn:matroid} below).  Every finite collection of vectors give rise to a matroid whose bases are its maximal independent sets, and indeed the definition of a matroid is a combinatorial abstraction of the idea of linear independence.  Given a matroid $M$ with ground set $[n]=\{1,2,\dots,n\}$ and rank~$r$, the \defterm{base polytope} $\Po_M$ is the convex hull in $\R^n$ of the indicator vectors of its bases, each of which contains $r$ ones and $n-r$ zeroes.
Matroid base polytopes provide geometric insight into the structure of matroids (see, e.g., \cite{FeichtnerSturmfels}).  In particular, each edge of $\Po_M$ corresponds to a pair of vertices with symmetric difference of minimum size two, so it is parallel to the difference of two standard basis vectors.  Thus, matroid base polytopes fall into the important class of polytopes known as \defterm{generalized permutohedra} \cite{Beyond,PRW}, and in fact they
are exactly the generalized permutohedra whose vertices have all coordinates equal to 0 or 1 \cite{GGMS}.
 
Here we sketch what is known about volumes and Ehrhart polynomials $\ehr_M(t)=\ehr_{\Po_M}(t)$ of matroid base polytopes.
One of the most basic cases is the \defterm{hypersimplex} $\Delta_{r,n}$, the cross-section of the unit cube $[0,1]^n\subset\R^n$ by the affine (geometric) hyperplane of points with coordinates that sum to an integer~$r$.  The hypersimplex is the base polytope of the uniform matroid $U_r(n)$, whose bases are \textit{all} the $r$-subsets of $[n]$.
It is classical that volumes of hypersimplices are given by the Eulerian numbers, which enumerate permutations by numbers of descents \cite[sequence A008292]{OEIS}.  Their Ehrhart polynomials were calculated by Katzman \cite[Cor.~2.2]{Katzman}.
Ferroni \cite[Thm.~4.3]{FerroniHypersimplices} gave a combinatorial formula for the coefficients arising in Katzman's formula in terms of Eulerian numbers and weighted Lah numbers, and proved that they are positive.

For general matroid polytopes, volume formulas were given by Ardila, Benedetti, and Doker \cite{ABD} and Ashraf~\cite{Ashraf}, but much less is known about their Ehrhart polynomials.
De~Loera, Haws, and K\"oppe \cite{DeLoeraHawsKoppe} conjectured that matroid polytopes are Ehrhart positive in general, and work of Castillo and Liu \cite{CastilloLiu} suggested that even generalized permutohedra might be Ehrhart positive.
However, recently Ferroni \cite{FerroniNotPositive} explicitly constructed matroids that are not Ehrhart positive.
Our work draws on and expands Ferroni's, so we will describe his technique in some detail.







A matroid of rank~$r$ is called \defterm{paving} if every circuit has cardinality greater than or equal to~$r$, and it is \defterm{sparse paving} if it and its dual are both paving.
Paving matroids are well-known objects to matroid theorists \cite[Chapter 2,3]{Welsh}.
It is conjectured and widely believed that asymptotically all matroids are paving matroids, or even sparse paving matroids; see \cite{mayhew}, \cite{pendavingh-vanderpol}, \cite[Chapter~15.5]{Oxley}.
Ferroni used the well-known fact that the base polytope of a sparse paving matroid $M$ could be obtained from a hypersimplex by slicing with (geometric) hyperplanes (see \cite{HerrmaannJoswig} for an application of this fact). Each piece sliced off in this way is itself a base polytope of the \defterm{minimal matroid} $T_{r,n}$ (so called because they have the least number of bases for their rank and ground set size among all such connected matroids), whose Ehrhart polynomials Ferroni had previously calculated and proven that they were positive in~\cite{Ferroni2021A}.
The result is an explicit formula for the Ehrhart polynomial of a sparse paving matroid, which Ferroni was able to show did not always have positive coefficients.

It is useful to turn the construction around and regard a hypersimplex as constructed from a paving or sparse paving matroid base polytope by attaching panhandle or minimal matroid polytopes.
The combinatorial analogue of this geometric operation is \defterm{circuit-hyperplane relaxation} \cite[p.39]{Oxley}, which adds a single basis to a matroid.
Ferroni\cite[Thm.~1.8]{Ferroni2021A} showed that if $\widetilde{M}$ is obtained from $M$ by relaxing a circuit-hyperplane, then \[\ehr_{\widetilde{M}}(t)=\ehr_M(t)+\ehr_{T_{r,n}}(t-1).\]
It is evident from Ferroni's explicit formula for $\ehr_{T_{r,n}}(t)$ in~\cite{Ferroni2021A} that the polynomial $\ehr_{T_{r,n}}(t-1)$ has nonnegative coefficients. Therefore, the operation of circuit-hyperplane relaxation preserves Ehrhart positivity.

Our work extends Ferroni's methods from sparse paving matroids to paving matroids. 
We show that the base polytope of any paving matroid $M$ (not necessarily sparse) can be obtained from a hypersimplex by slicing with (geometric) hyperplanes.
The pieces sliced off are base polytopes of a class that we call \defterm{panhandle matroids} (see Figure~\ref{fig:panhandle} for the meaning of this terminology).  We calculate the Ehrhart function of a panhandle matroid in Proposition~\ref{prop:Ehrhart-panhandle}, then give explicit polynomial formulas in Theorem~\ref{thm:betterPanhandleformula} and Corollary~\ref{formula for positivity}.
This ``slicing" or ``sculpting" approach has been previously used in the study of polytopes (e.g., \cite{HerrmaannJoswig, JoswigSchroter, Kim}), but has only recently been used to study the lattice-point enumeration of lattice polytopes (e.g., \cite{FerroniMcGinnis, FerroniJochemkSchroter, FerroniSchroter}).
For matroid polytopes, slicing corresponds to a generalization of relaxation called \defterm{stressed-hyperplane relaxation}, first introduced in~\cite{relaxation}: every paving matroid can be transformed into a uniform matroid by iteratively relaxing stressed hyperplanes.
We calculate the effect of stressed-hyperplane relaxation on the Ehrhart polynomial (Theorem~\ref{thm:Ehrhart-relaxation-improved}) and use it to calculate the Ehrhart polynomial of a paving matroid (Theorem~\ref{thm:Ehrhart-paving-improved-formula}).
Panhandle matroids are instances of \defterm{Schubert matroids}, which are isomorphic to  \defterm{lattice path matroids} of (non-skew) Ferrers diagrams \cite{LPM}; work on the Ehrhart theory of Schubert matroids includes \cite{Bidkhori,FanLi}.  We conjecture that panhandle matroids are Ehrhart positive in Conjecture~\ref{PanhandlePositive} and outline an approach to prove this conjecture, which leads to a sufficient condition in Conjecture~\ref{big conjecture}.

The Ehrhart polynomial of a lattice polytope $P$ of dimension $n$ can always be written in the form $\ehr_P(t)=\sum_{i=0}^nh_i^*\binom{t+n-i}{n}$; the sequence $(h_0^*,\dots,h_n^*)$ is called the \defterm{$h^*$-vector}. 
Equivalently, $\sum_{t\geq 0}\ehr_P(t)z^t=\frac{h^*(P;z)}{(1-z)^{n+1}}$, where $h^*(P;z)=h_0^*+h_1^*z+\cdots+h_n^*z^n$.
Beyond Ehrhart positivity for matroid polytopes, De Loera et al.\ conjectured that the $h^*$-vector of any matroid polytope is unimodal \cite{DeLoeraHawsKoppe}.  
This conjecture seems very hard, although it has been verified for minimal matroids by Knauer, Mart\'inez-Sandoval, and Ram\'irez Alfons\'in \cite[Thm.~4.9]{KMR} and for sparse paving matroids of rank-2 by Ferroni, Jochemko and Schr\"oter~\cite{FerroniJochemkSchroter}. 
We hope to study the $h^*$-vectors of paving and panhandle matroids in a future article.


The paper is structured as follows.
Sections~\ref{sec:prelims}--\ref{sec:panhandle} include background material on Ehrhart theory and matroid base polytopes, stressed-hyperplane relaxation and its relevance to paving matroids, and panhandle matroids.
Along the way, we describe a more general version of relaxation in paving matroids (Proposition~\ref{prop:general-relaxing-for-paving}) that merits further study.

In Section~\ref{sec:ehrhart}, we state and prove the main results on the Ehrhart polynomials of panhandle matroids, stressed-hyperplane relaxations, and paving matroids, and show how Ferroni's formula for sparse paving matroids arises as a special case.

In Section~\ref{sec:ehrhart-panhandle}, we propose two closely related conjectures: panhandle matroids are Ehrhart positive (Conjecture~\ref{PanhandlePositive}) and stressed-hyperplane relaxation preserves Ehrhart positivity (Conjecture~\ref{RelaxationPositive}).
The conjectures reduce to showing that two near-identical polynomials $\phi_{r,s,n}(t)$ and $\tilde{\phi}_{r,s,n}(t)$, defined in~\eqref{phisrlt} and~\eqref{tildephi} respectively, have positive coefficients. 
Both statements appear to be true based on computational evidence, as we will explain later.
Moreover, positivity of $\phi_{r,s,n}(t)$ reduces to a combinatorial statement (Conjecture~\ref{big conjecture}) involving refinements of the weighted Lah numbers introduced by Ferroni \cite{FerroniHypersimplices}; again, there is significant computational evidence that this statement is true.
In Section~\ref{sec:weighted Lah}, we give a purely combinatorial proof of a formula for the weighted Lah numbers (Theorem~\ref{weighted lah}), which was originally proven using generating function methods by Ferroni, with the hope that the argument can be extended to attack Conjecture~\ref{big conjecture}. 


In Section~\ref{sec:volumes}, we apply Ashraf's volume formula to give closed formulas for the volumes of the base polytopes of panhandle matroids (Theorem~\ref{thm:volofpan}), stressed-hyperplane relaxations (Theorem~\ref{thm:volume-relaxation}), and paving matroids (Theorem~\ref{thm:volume-paving}).

Section~\ref{subsec:steiner_systems} applies our general results to paving matroids constructed in a standard way from Steiner systems and projective planes.
In particular, a consequence of our work is that the matroid polytopes arising from two Steiner systems with the same combinatorial parameters have the same Ehrhart polynomials, although they need not be isomorphic as polytopes.

\section{Background and preliminaries}\label{sec:prelims}

The symbol $\N$ will denote the nonnegative integers, and we write $[n]=\{1,\dots,n\}$. For integers $n$ and $k$, we adopt the convention that $\binom{n}{k}=0$ whenever $k<0$ or $k>n$. For a polynomial $f(t)$, we may also consider $\binom{f(t)}{k}$ to be the polynomial
\[
\frac{f(t)(f(t)-1)\cdots (f(t)-k+1)}{k!}.
\]
We note that, for a given value $t_0$, if $f(t_0)$ is a nonnegative integer, then the two interpretations of $\binom{f(t_0)}{k}$ are consistent.

\subsection{Polytopes and Ehrhart theory}

A \defterm{polytope} is the convex hull of a finite set of points in~$\R^n$, or alternatively the set of solutions of a finite set of linear equalities and inequalities, provided it is bounded.
The equivalence of the two definitions is a foundational result in polytope theory.
Standard references about polytopes are \cite{Grunbaum} and \cite{Ziegler}.

The \defterm{dimension} of a polytope $P\subset\R^n$ is $\dim P=\dim \operatorname{aff}(P)$, where $\operatorname{aff}(P)$ is the smallest affine subspace of $\R^n$ containing $P$.
The \defterm{normalized volume} or \defterm{relative volume} $\Vol(P)$ is the volume with respect to the lattice $\Z^n\cap \operatorname{aff}(P)$, 
The \defterm{Ehrhart function} of $P$ is defined as
\[\ehr_P(t): = |tP\cap\Z^n|\] where $tP=\{t\mathbf{x}:\ \mathbf{x}\in P\}$.
When all vertices of $P$ have integer coordinates (the only case we will consider), the Ehrhart function is a polynomial in $t$ of degree equal to the dimension of $P$, with leading coefficient $\Vol(P)$ \cite{BeckRobins, Ehrhart}.

\subsection{Matroids}\label{subsec:matroids}
We briefly review the definition of a matroid and relevant terminology.  
Standard sources include~\cite{Oxley} and~\cite{Welsh}.

\begin{definition} \label{defn:matroid}
Let $E$ be a finite set.
A \defterm{matroid basis system on ground set $E$} is a nonempty family $\B \subseteq 2^E$ of \defterm{bases} satisfying the \defterm{exchange axiom}: for all distinct $B,B'\in\B$ and all $e\in B\sm B'$, there exists $e'\in B'\sm B$ such that $\left(B\sm \{e\}\right)\cup\{e'\}\in \B$.
The pair $M=(E,\B)$ defines a \defterm{matroid}. 
Any subset of a basis is called an \defterm{independent set}.
The \defterm{rank function} of $M$ is defined by $\rank(A)=\max\{|A\cap B|:\ B\in\B\}$ for $A\subseteq E$.
The number $\rank(E)$ is called the \defterm{rank} of $M$, often abbreviated~$r$.
\end{definition}
The family of independent sets contains the same information as the basis system, as does the rank function, so a matroid can be defined by specifying any of these objects.
Some additional matroid terminology that will be useful:
\begin{itemize}
\item A \defterm{circuit} in a matroid is a minimal dependent subset of $E$.
A \defterm{loop} is a circuit of size~1.
\item A \defterm{flat} of $M$ is a subset $F\subseteq E$ such that $\rank(G)>\rank(F)$ for every $G\supsetneq F$.
A \defterm{hyperplane} is a flat of rank $r-1$. 
We write $\HH$ or $\HH_M$ for the set of hyperplanes of $M$.
\item The \defterm{direct sum} of matroids $M=(E,\B)$ and $M'=(E',\B')$ on disjoint ground sets is the matroid $M\oplus M'$ on $E\cup E'$  with basis system $\{B\cup B':\ B\in \B, B'\in \B'\}$.
\item A matroid is \defterm{connected} if it cannot be written as a direct sum of two non-empty matroids.
Every matroid admits a unique decomposition as a direct sum of connected matroids, called its \defterm{components.}
\item A \defterm{paving matroid} is a matroid $M$ for which every circuit has cardinality at least $\rank(M)$.
\item A \defterm{circuit-hyperplane} is a set that is both a circuit and a hyperplane.  If $C$ is a circuit-hyperplane, then the family $\B\cup\{C\}$ is in fact a matroid basis system~\cite[Prop.1.5.14]{Oxley}, called the \defterm{relaxation of $M$ at $C$}.
\end{itemize}
This use of the word ``hyperplane'' is standard in matroid theory but unfortunately conflicts with the geometric use of ``hyperplane,'' which we will also need.  It should be clear from context which meaning is intended. 

As a standard example, let $r$ be a nonnegative integer.
The \defterm{uniform matroid} $U_{r,n}$ on ground set $[n]$ has basis system $\binom{[n]}{r}=\{B\subseteq [n]:\ |B|=r\}$, independence system $\{B\subseteq[n]:\ |B|\leq r\}$, and rank function $\rank(A)=\min(r,|A|)$.

\subsection{Matroid base polytopes}\label{subsec:matroid_base_polys}
Every matroid has an associated polytope called its \defterm{base polytope}, which contains the same information as the basis system, rank function, etc., but enables the matroid to be studied geometrically.
The study of matroid polytopes dates back to Edmonds \cite{Edmonds} and also appears in the context of combinatorial optimization \cite{CookCunninghamPulleyblankSchrivjer, Fujishige}.
In particular, matroid base polytopes are a well-understood subclass of \textit{generalized permutohedra}, hence significant from the point of view of optimization and combinatorial Hopf theory; see
\cite{Fujishige,AguiarArdila,PRW,GGMS}.
A good starting reference for the geometry of matroid base polytopes is \cite[Sec.~2]{FeichtnerSturmfels}.

In what follows, let $ \conv A $ denote the convex hull of a point set $A\subseteq\R^n$.

\begin{definition} \label{def:matroid_polytopes}
The \textbf{base polytope} of a matroid $M$ is \[\Po_M:=\conv\{\ee_B: B\in\B\} \subset \R^{E},\] where $\displaystyle \ee_B=\sum_{i\in B}\ee_i$ and $\ee_i$ is the $i$th standard basis vector for $\R^E$.
\end{definition}
The polytope $\Po_M$ is always contained in the (geometric) hyperplane $\{\bx=(x_i)\in\R^E:\ \sum x_i=r\}$, hence must have dimension strictly less than $|E|$.
In fact, $\dim\Po_M=|E|-c$, where $c$ is the number of components.
It has the following description in terms of inequalities \cite[Prop.~2.3]{FeichtnerSturmfels}:
\begin{equation}\label{feichtner}
\Po_M=\bigg\{(x_1,\dots, x_n)\in \R_{\geq0}^n\ :\ \sum_{i=1}^n x_i = \rank(M),\ \sum_{i\in F} x_i \leq \rank(F) \textrm{ for all flats } F\bigg\}.
\end{equation}

\begin{example}
The \defterm{hypersimplex} $\Delta_{r,n}$ is the base polytope of the uniform matroid $U_{r,n}$.
It is the convex hull of the $\binom{n}{r}$ points in $\R^n$ whose coordinates consist of $r$ 1's and $n-r$ 0's, or equivalently the intersection of the unit cube $[0,1]^n$ with the (geometric) hyperplane $\{\bx\in\R^n:\ x_1+\cdots+x_n=r\}$.
It is a point when $r=0$ or $r=n$ and a simplex when $r=1$ or $r=n-1$.
The next simplest example is the polytope
\[\Po_{U_{2,4}}=\conv\{(1,1,0,0), (1,0,1,0), (1,0,0,1), (0,1,1,0),(0,1,0,1),(0,0,1,1)\},\]
which is an octahedron in which points with disjoint supports are antipodal.
The geometric properties of the hypersimplex are well understood.
Its normalized volume  is
\begin{equation} \label{eq:volume-hypersimplex}
\Vol(\Delta_{r,n})=A(n-1,r-1),
\end{equation}
where $A(n-1,r-1)$ is the \defterm{Eulerian number}, which counts permutations of $[n-1]$ with $r-1$ descents \cite[sequence A008292]{OEIS} (see also \cite[pp.32--35]{Stanley}).
Katzman~\cite[Cor.~2.2]{Katzman} (see also \cite[Problem~4.62]{Stanley}) computed the Ehrhart polynomial of the hypersimplex as
\begin{equation} \label{eq:Katzman}
\ehr_{U_{r,n}}(t)=\ehr_{\Delta_{r,n}}(t)=\sum_{j=0}^{r-1}(-1)^j\binom{n}{j} \binom{(r-j)t-j+n-1}{n-1}.
\end{equation}
\end{example}

More general formulas for the volume of a matroid polytope were given by Ardila, Benedetti and Doker \cite{ABD} and Ashraf \cite{Ashraf}; we will describe Ashraf's formula in detail in Section~\ref{sec:volumes}.

\section{Relaxation in matroids}\label{sec:relax}

Throughout this section, let $M$ be a matroid with ground set $E$, rank $r$, and basis system $\B$.
In \cite{relaxation}, the authors introduced a generalized notion of circuit-hyperplane relaxation, as we now describe.

\begin{definition}{\cite[Def.~3.1]{relaxation}}\label{def:stressed}
A hyperplane $H$ of $M$ is a \defterm{stressed hyperplane} if every subset of $H$ of size $r$ is a circuit.
\end{definition}

It follows from the definitions in Section~\ref{subsec:matroids} that every circuit-hyperplane is stressed.

Stressed hyperplanes exist in disconnected matroids only in extreme cases:

\begin{proposition}\label{prop:disconnected_stressed}
Let $ M=M_1\oplus M_2 $ be a rank $ r $ disconnected matroid on $ n $ elements with a stressed hyperplane $ H $ of cardinality $ s \geq 1$.
Then, $ M \cong U_{r-1, s}\oplus U_{1, n-s} $.
\end{proposition}
\begin{proof}
First, if $r=1$, then the unique stressed hyperplane is the set of loops, from which the conclusion follows.  
Henceforth, assume $r\geq2$.  
In this case, note that $ M $ must be loopless (since a hyperplane contains every loop, but a stressed hyperplane can contain no loops), so $M_1$ and $M_2$ are loopless as well.  
In particular, their ranks are both at least 1 and at most $ r-1 $.
Without loss of generality, we may assume that $H = E_1 \cup H_2$, where $ E_1 $ is the ground set of $ M_1 $ and $ H_2 $ is a hyperplane of $ M_2 $.
Since $ H $ is stressed, every subset of $E_1$ of size $r-1$ is independent, so $ M_1 $ must be uniform of rank $ r-1 $.
Therefore $ M_2 $ is a rank-1 loopless matroid, hence uniform.  
In particular the only hyperplane of $ M_2 $ is the empty set, so $ H = E_1 $, which implies $ M_1 \cong U_{r-1, s} $ and $ M_2 \cong U_{1, n-s} $.
\end{proof}

Moreover, stressed hyperplanes are intimately connected with paving matroids:

\begin{proposition}{\cite[Prop.~3.16]{relaxation}}\label{prop:stressed_hyps_paving}
A matroid is a paving matroid if and only if every hyperplane is stressed.
\end{proposition}

Compare this result with \cite[Prop.~2.1.21]{Oxley}: a family $\HH$ of subsets of $E$, all of size at least $r-1$, is the set of hyperplanes of a paving matroid of rank~$r$ if and only if each $(r-1)$-subset of $E$ is contained in exactly one element of $\HH$.

\begin{definition}\label{def:gen_relaxed}
Let $S\subseteq E$ be a set containing no basis.  
We say that $S$ \defterm{can be relaxed} if $\Rel_S(\B):=\B\cup\binom{S}{r}$ is a matroid basis system.
In this case we call the resulting matroid the \textbf{relaxation of $M$ at $S$}, denoted by $\Rel_S(M)$.
\end{definition}

\begin{proposition}\label{prop:relaxation}\cite[Thm.~1.2]{relaxation}
If $H$ is a stressed hyperplane of $M$, then $H$ can be relaxed. 
\end{proposition}

When $S$ is a circuit-hyperplane, the relaxation of $M$ at $S$ coincides with the usual notion of a relaxation as in \cite[Section 1.5]{Oxley}, and when $S$ is a stressed hyperplane, we recover the stressed-hyperplane relaxation (Proposition~\ref{prop:relaxation}). 

In our study of matroid base polytopes, the sets we relax will always be stressed hyperplanes. 
On the other hand, it is possible to relax other sets in matroids.  
In the remainder of this section, we describe a generalization of stressed-hyperplane relaxation for further study.

\begin{example}\label{ex:gen_rel}
Consider the matroid with basis system $\B=\{\{1,2,3\},\{1,2,4\}\}$.
The set $S=\{1,3,4\}$ is a hyperplane, but it is not stressed since its subset $\{3,4\}$ is dependent.
Nevertheless, $S$ can be relaxed to produce the matroid with basis system $\{\{1,2,3\},\{1,2,4\},\{1,3,4\}\}$.
\end{example}

As the wording in Definition \ref{def:gen_relaxed} suggests, there exist sets that cannot be relaxed.
The following provides some necessary conditions.

\begin{proposition} \label{pillow}
Let $S \subseteq E$ be a subset of size at least $r$ containing no basis of $M$.
Then, $S$ can be relaxed only if $\rank(S)=r-1$ and $S$ is not properly contained in any stressed hyperplane of $M$.
\end{proposition}

\begin{proof}
Let $\widetilde{\mathcal{B}}=\mathcal{B}\cup\binom{S}{r}$.
Suppose $\rank(S)<r-1$ and let $T\subseteq S$ be a maximal independent subset.
Note that $|T|<r-1$.
Let $B$ be a basis of $M$ with $T\subseteq B$ and let $\widetilde{B} \in \binom{S}{r}$ with $T\subseteq \widetilde{B}$.
Notice that $B \cap S = T$ because $B\cap S$ is independent in $M$ and $T$ is a maximal independent subset of $S$.
Let $x\in \widetilde{B} \sm B$.
Then $\widetilde{B} \sm \{x\}$ is dependent in $M$ as it is a size $r-1$ subset of $S$.
Therefore, for any $y \in B \sm \widetilde{B}$, if $(\widetilde{B}\sm\{x\}) \cup \{y\} \in \widetilde{\mathcal{B}}$, then it must be the case that $(\widetilde{B}\sm\{x\})\cup \{y\} \in \binom{S}{r}$.
This would imply that $y\in S$ which is impossible because $B\cap S = T$, but $y\not\in T\subseteq \widetilde{B}$.
Thus, the basis exchange axiom fails for $\widetilde{\mathcal{B}}$ in this case.

Now suppose $S$ has rank $r-1$ and is properly contained in a stressed hyperplane $H$ of $M$.
Then there exists $h\in H\sm S$.
Let $T \in \binom{S}{r-2}$ and let $B$ be a basis of $M$ containing $T \cup \{h\}$.
Note that such a basis $B$ must exist because $T \cup \{h\}$ is a size $r-1$ subset of the stressed hyperplane $H$ and is therefore independent in $M$.
Furthermore, there exists a unique element $x \in B \sm H$.
Then, let $\widetilde{B} \in \binom{S}{r}$.
Notice that $x \in B \sm \widetilde{B}$ because $\widetilde{B} \subseteq H$.
For any $y \in \widetilde{B} \sm B$, we have $h \in (B \sm \{x\}) \cup \{y\} \subseteq H$.
This means that $(B \sm \{x\}) \cup \{y\} \not\in \mathcal{B}$ as $H$ contains no bases of $M$ and also $(B\sm \{x\}) \cup \{y\} \notin \binom{S}{r}$ since $h \notin S$.
Thus, $(B\sm \{x\}) \cup \{y\} \notin \widetilde{\mathcal{B}}$ so the basis exchange axiom fails for $\widetilde{\mathcal{B}}$ in this case too.
\end{proof}

Applying Proposition~\ref{pillow} to paving matroids (where every hyperplane is stressed) and combining with Proposition~\ref{prop:relaxation}, we obtain the following:

\begin{proposition}\label{prop:general-relaxing-for-paving}
Let $M$ be a rank $r$ paving matroid with ground set $E$ and basis system $\mathcal{B}$.
Let $S \subseteq E$ be a subset of size at least $r$ containing no basis of $M$.
Then, $S$ can be relaxed if and only if $S$ is a hyperplane of $M$.
\end{proposition}


\section{Panhandle matroids}\label{sec:panhandle}

\begin{definition}
Let $r\leq s<n$ be nonnegative integers.
The \defterm{panhandle matroid} $\Pan_{r,s,n}$ is the rank-$r$ matroid on ground set $[n]$ with basis system \[\B=\B_{r,s,n}=\left\{B\in\binom{[n]}{r}:\ |B\cap[s]|\geq r-1\right\}.\]
\end{definition}

One can check directly that $\B_{r,s,n}$ satisfies the axioms of a matroid basis system.
Alternatively, one can observe that $\B_{r,s,n}$ is a \textit{lattice path matroid}~\cite{LPM}, as we now explain.
Consider the Ferrers diagram shown in Figure~\ref{fig:panhandle} (from which the name ``panhandle'' derives).
Given a lattice path from $(0,0)$ to $(n-r,r)$ that stays in the Ferrers diagram, label its steps sequentially $1,\dots,n$.
The sets of North steps arising from such paths are precisely the elements of $\B_{r,s,n}$, hence form a matroid basis system.

\begin{center}
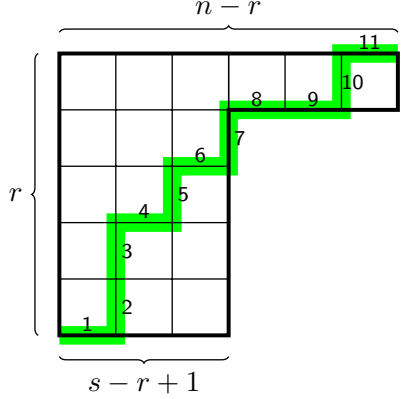
\begin{figure}[ht]
\begin{tikzpicture}[scale=0.75, line width=.5pt]
\draw[line width=7pt,green] (0,0)--(1,0)--(1,2)--(2,2)--(2,3)--(3,3)--(3,4)--(5,4)--(5,5)--(6,5);
\foreach \x/\y/\l in {0.5/0.2/1, 1.2/0.5/2, 1.2/1.5/3, 1.5/2.2/4, 2.2/2.5/5, 2.5/3.2/6, 3.2/3.5/7, 3.5/4.2/8, 4.5/4.2/9, 5.2/4.5/10, 5.5/5.2/11}
	\node at (\x,\y) {\scriptsize\sf\l};
\draw (3,4) grid (6,5);
\draw (0,0) grid (3,5);
\draw[decoration={brace,mirror, raise=8pt},decorate]
 (0,0) -- node[below=10pt] {$s-r+1$} (3,0);
\draw[decoration={brace,mirror, raise=8pt},decorate]
 (6,5) -- node[above=10pt] {$n-r$} (0,5);
\draw[decoration={brace,mirror, raise=8pt},decorate]
 (0,5) -- node[left=10pt] {$r$} (0,0);
\draw[line width=1.5pt](0,0)--(3,0)--(3,4)--(6,4)--(6,5)--(0,5)--cycle;
\end{tikzpicture}
\caption{The panhandle matroid $\Pan_{5,7,11}$ as a lattice matroid.  The lattice path shown in green gives rise to the basis $\{2,3,5,7,10\}$.}\label{fig:panhandle}
\end{figure}
\end{center}

The panhandle matroid is connected by \cite[Theorem~3.6]{LPM}.
Special cases include the \textit{minimal matroids} of \cite{Ferroni2021A}, which are equivalent to panhandle matroids $\Pan_{r,r,n}$ (where the Ferrers diagram is a hook shape), and the uniform matroids $\Pan_{r,n-1,n}\cong U_{r,n}$.

The Ehrhart theory of lattice path matroid polytopes has been studied in~\cite{Bidkhori,KMR}.  At present, there is no formula known for their Ehrhart polynomials (although a non-polynomial formula for their Ehrhart functions appears in~\cite{Bidkhori}).  We will give a polynomial formula for the special case of panhandle matroids, which will be a key ingredient in the formula for Ehrhart polynomials of paving matroids.

\begin{proposition}\label{prop:rank_panhandle}
Let $\rank: 2^{[n]}\to \N$ be the rank function for $\Pan_{r,s,n}$ and $T\subseteq [n]$. Let $T_1:=T\cap  [s]$ and $T_2:=T\cap [s+1,n]$.
Then

\[\rank(T)=\begin{cases}
\min(|T_1|,r) & T_2=\emptyset\\
\min(|T_1|+1,r) & \text{otherwise}
\end{cases}.\]
\end{proposition}
\begin{proof}

If $T_2=\emptyset$, then $T\subseteq [s]$.
When $|T|\leq r$, then note that $T$ is independent since it contains no elements of $[s+1,n]$, and so $\rank(T)=|T|$.
If instead $|T|> r$, then $T$ must contain a basis for $\Pan_{r,s,n}$ and so $\rank(T)=r$.

Suppose otherwise that $T_2\neq \emptyset$.
Notice that for $|T_1|\leq r-1$ and any $x\in [s+1,n]$, $T_1\cup \{x\}$ is an independent set.
Then \[\rank(T_1\cup T_2)\geq |T_1|+1.\]
However, we also have $\rank(T_1\cup T_2)\leq |T_1|+1$ as the elements of $[s+1,n]$ (and hence $T_2$) are dependent.
Hence, $\rank (T)=|T_1|+1$.
On the other hand, if $|T_1|\geq r-1$, then $\rank(T)=r$, because $T$ is either a basis (if $|T_1|=r-1$) or contains a basis (if $|T_1|>r-1$).

\end{proof}

We use our understanding of the rank function for panhandle matroids to classify the flats.
\begin{proposition}\label{prop:flats_for_panhandle}
The flats with rank less than $r$ of $\Pan_{r,s,n}$ are of two types:
\begin{itemize}
    \item {Type 1:} Subsets of $[s]$ of size at most $r-1$.
    \item {Type 2:} $[s+1,n]\cup A$ where $A\subseteq [s]$ is a set of size at most $r-2$.
\end{itemize}
\end{proposition}
\begin{proof}

First, suppose that $F\subseteq [s]$ is a set of cardinality at most $r-1$.
If $x\in [n]\sm F$, observe that $\rank(F\cup x)=|F|+1>|F|=\rank(F)$ by Proposition \ref{prop:rank_panhandle}.
Hence, $F$ must be a flat.
On the other hand, if $|F|\geq r$, note that $\rank(F)=r$, and hence $F$ could not be a flat of rank less than $r$.

Now suppose that $F\nsubseteq [s]$, and so $F\cap [s+1,n]\neq \emptyset$.
Observe that $[s+1,n]$ is a dependent set.
Hence, any flat containing an element from $[s+1,n]$ must contain all of $[s+1,n]$.
To this end, if $F$ is to be a flat, we may assume that $F=[s+1,n]\cup A$, where $A \subseteq [s]$.
For $F$ to have rank less than $r$, we must also have that $|A|\leq r-2$ by Proposition \ref{prop:rank_panhandle}.
Thus, for any $x\in [n]\sm F$, we have $A\subsetneq A\cup \{x\}\subseteq [s]$.
Combining this with Proposition \ref{prop:rank_panhandle}, we get \[\rank(([s+1,n]\cup A)\cup\{y\})= |A\cup \{y\}|+1>|A|+1= \rank([s+1,n]\cup A).\]
That is, $F$ is a flat.
\end{proof}

Combining Proposition \ref{prop:flats_for_panhandle} with~\eqref{feichtner} yields the following.

\begin{proposition}\label{Pan-characterization}
The polytope $\Po_{\Pan_{r,s,n}}$ is given by
\[
\Po_{\Pan_{r,s,n}}=\{(x_1,\dots, x_n)\in [0,1]^n\ :\ \sum_{i=1}^nx_i=r \textrm{ and } \sum_{i=s+1}^n x_i \leq 1\}.
\]
\end{proposition}

\begin{proof}
We show that the inequality
\begin{equation}\label{eq:pan_poly_condition}
    \displaystyle\sum_{i\in F} x_i \leq \rank(F)
\end{equation} in~\eqref{feichtner} reduces to
\begin{equation}\label{eq:new_pan_poly_condition}
    \displaystyle \sum_{i=s+1}^n x_i \leq 1
\end{equation}

Recall by Proposition \ref{prop:flats_for_panhandle} that $\Pan_{r,s,n}$ has two types of flats.
The type-1 flats are the subsets of $[s]$ with cardinality at most $r-1$, which are also independent.
Hence, if $F$ is a type-1 flat, inequality \eqref{eq:pan_poly_condition} is superfluous since  $x_i\leq 1$ for all $i$ and $\rank(F)=|F|$.

The type-2 flats are the sets of the form $[s+1,n]\cup A$ where $A\subseteq [s]$ is a set of cardinality at most $r-2$.
In this case, if $F$ is such a flat, recall from Proposition \ref{prop:rank_panhandle} that $\rank(F)=|A|+1$.
Hence, inequality~\eqref{eq:pan_poly_condition} becomes
\[
\sum_{i=s+1}^nx_i+\sum_{i\in A}x_i\leq |A|+1,
\]
Note that $(x_1,\dots, x_n)$ satisfies this inequality if and only if
\[\sum_{i=s+1}^nx_i\leq 1,\]
since the requirement that $x_i\leq 1$ for all $i$ implies that $\displaystyle \sum_{i\in A}x_i\leq |A|$.
Hence, inequality~\eqref{eq:pan_poly_condition} reduces to inequality~\eqref{eq:new_pan_poly_condition}.
\end{proof}


\section{Proofs of the main theorems} \label{sec:ehrhart}

Recall our conventions on binomial coefficients: if $n$ and $k$ are integers, then $\binom{n}{k}=0$ whenever $k<0$ or $k>n$, and when $f(t)$ is a polynomial, we put $\binom{f(t)}{k}=f(t)(f(t)-1)\cdots (f(t)-k+1)/k!$.

\subsection{Ehrhart polynomials of panhandle matroids}

The core of the calculation for panhandle matroids is the following result about counting integer solutions to certain linear equations.

\begin{lemma}\label{lemma:intsolutions}
Fix $s,t,r,m$ with $0\leq r \leq s$ and $0\leq m \leq t$.
The number of nonnegative integer solutions to $\sum_{j=1}^sx_j = tr-m$, with $0\leq x_j\leq t$, is
\begin{equation}\label{solutions}
\sum_{i=0}^{s-r} (-1)^i \binom{s}{i} \binom{t(s-r-i)+m+s-1-i}{s-1}.
\end{equation}
\end{lemma}

\begin{proof}
Setting $y_j=t-x_j$, we see that it is equivalent to count integer solutions to the equation
\begin{equation} \label{solns-y-form}
\sum_{j=1}^sy_j = ts - (tr-m)=t(s-r)+m
\end{equation}
with $0\leq y_j\leq t$ for all $j$. 
For $I\subseteq [s]$, we claim that the number of nonnegative solutions to~\eqref{solns-y-form} for which $y_j>t$ if $j\in I$ is $\binom{t(s-r-i)+m+s-1-i}{s-1}$.
Indeed, setting $z_j=y_j-(t+1)$ for $j\in I$ and $z_j=y_j$ otherwise. 
Then, letting $i=|I|$, equation~\eqref{solns-y-form} becomes
\[\sum_{j=1}^sz_j = t(s-r)+m-i(t+1)=t(s-r-i)+m-i\]
which has $\binom{t(s-r-i)+m+s-1-i}{s-1}$ nonnegative solutions. Note that this binomial coefficient is 0 if $i>s-r$, due to the assumption $m\leq t$.
Now~\eqref{solutions} follows by inclusion/exclusion. 
\end{proof}

As pointed out by an anonymous referee, the result of Lemma~\ref{lemma:intsolutions} may hold in greater generality (e.g., relaxing the condition $m\leq t$).  However, imposing this restriction makes the proof easier, and we will solely be concerned with the case $m\leq t$ in what follows.

\begin{proposition}\label{prop:Ehrhart-panhandle}
The value of the Ehrhart function of the polytope $\Po_{\Pan_{r,s,n}}$ for given positive integer $t$ is
\[
    \ehr_{{\Pan_{r,s,n}}}(t)=\sum_{m=0}^t\sum_{i=0}^{s-r}(-1)^i\binom{s}{i}\binom{t(s-r-i)+m+s-1-i}{s-1}\binom{m+n-s-1}{m}.
\]
\end{proposition}

\begin{proof}
By Proposition \ref{Pan-characterization}, we have the following.
\begin{align*}
\ehr_{{\Pan_{r,s,n}}}(t)&=\#(\mathbb{Z}^n\cap t\Po_{\Pan_{r,s,n}})\\
&=\#\left\{x\in [0,t]^n\ :\ \sum_{i=1}^n x_i =tr,\ \sum_{i=s+1}^n x_i\leq t\right\}\\
&=\sum_{m=0}^t \#\left\{x\in [0,t]^n\ :\ \sum_{i=1}^n x_i =tr,\ \sum_{i=s+1}^n x_i=m\right\}\\
&=\sum_{m=0}^t \#\left\{x\in [0,t]^n\ :\ \sum_{i=1}^{s} x_i =tr-m,\ \sum_{i=s+1}^n x_i=m\right\}\\
&=\sum_{m=0}^t\sum_{i=0}^{s-r}(-1)^i\binom{s}{i}\binom{t(s-r-i)+m+s-1-i}{s-1}\binom{m+n-s-1}{m}
\end{align*}
The last equality follows from Lemma \ref{lemma:intsolutions} and a standard stars-and-bars argument.
\end{proof}

The previous formula cannot be considered a polynomial in the dilation factor~$t$, which appears as a limit of summation.  
The next result rewrites the formula as a genuine polynomial in~$t$. 
In Theorem \ref{thm:betterPanhandleformula} below, we consider the right side of equation (\ref{eq:FirstEhrhartPoly}) to be a polynomial in $t$ by interpreting each binomial term containing $t$ therein as a polynomial.
Since $\ehr_{{\Pan_{r,s,n}}}(t)$ is given by a polynomial in $t$ by Ehrhart's Theorem, it suffices to show that each side of equation (\ref{eq:FirstEhrhartPoly}) agree for infinitely many values of $t$ to show that they agree as polynomials. In particular, it suffices to show that they agree for positive integer values of $t$.
Thus, we implicitly assume that $t$ is a positive integer in the proof of Theorem \ref{thm:betterPanhandleformula}, and we do this in several proofs throughout unless otherwise stated. 
Since the upper term of each binomial term in the proof of Theorem \ref{thm:betterPanhandleformula} is nonnegative for each positive integer value of $t$, there should be no confusion between the two previously mentioned interpretations of the binomial terms.



\begin{theorem}\label{thm:betterPanhandleformula}
The Ehrhart polynomial for the polytope $\Po_{\Pan_{r,s,n}}$ is
\begin{equation}\label{eq:FirstEhrhartPoly}
\ehr_{{\Pan_{r,s,n}}}(t)=\sum_{i=0}^{s-r}(-1)^i\binom{s}{i}\sum_{\ell=0}^{s-1}\binom{t(s-r-i)+s-1-i}{s-1-\ell}\binom{t+n-s}{n-s}\binom{t}{\ell}\frac{n-s}{n-s+\ell}.
\end{equation}
\end{theorem}

\begin{proof}
We start by rewriting Proposition~\ref{prop:Ehrhart-panhandle} using the Chu-Vandermonde identity \cite[eq.~(5.22), p.~169]{concretemath}:
\begin{align*}
    \ehr_{\Pan_{r,s,n}}(t)&=\sum_{m=0}^t\sum_{i=0}^{s-r}(-1)^i\binom{s}{i}\binom{t(s-r-i)+m+s-1-i}{s-1}\binom{m+n-s-1}{m}\\
   &=\sum_{m=0}^t\sum_{i=0}^{s-r}(-1)^i\binom{s}{i}\sum_{\ell=0}^{s-1}\binom{t(s-r-i)+s-1-i}{s-1-\ell}\binom{n-s-1+\ell}{\ell}\binom{m+n-s-1}{m-\ell}
\intertext{(using the identity $\binom{K}{R}\binom{N}{K}=\binom{N-R}{K-R}\binom{N}{R}$); note that $n-s-1\geq0$)}
    &=\sum_{m=0}^t\sum_{i=0}^{s-r}(-1)^i\binom{s}{i}\sum_{\ell=0}^{s-1}\binom{t(s-r-i)+s-1-i}{s-1-\ell}\binom{n-s-1+\ell}{\ell}\binom{m+n-s-1}{n-s-1+\ell}\\
    &=\sum_{i=0}^{s-r}(-1)^i\binom{s}{i}\sum_{\ell=0}^{s-1}\binom{t(s-r-i)+s-1-i}{s-1-\ell}\binom{n-s-1+\ell}{\ell}\sum_{m=0}^t\binom{m+n-s-1}{n-s-1+\ell}\\
    &=\sum_{i=0}^{s-r}(-1)^i\binom{s}{i}\sum_{\ell=0}^{s-1}\binom{t(s-r-i)+s-1-i}{s-1-\ell}\binom{n-s-1+\ell}{\ell}\binom{t+n-s}{n-s+\ell}
\intertext{(using the ``hockey-stick identity'' $\sum_{I=R}^N \binom{I}{R} = \binom{N+1}{R+1}$ \cite[p.~160]{concretemath})}
    &=\sum_{i=0}^{s-r}(-1)^i\binom{s}{i}\sum_{\ell=0}^{s-1}\binom{t(s-r-i)+s-1-i}{s-1-\ell}\binom{t+n-s}{n-s}\binom{t}{\ell}\frac{n-s}{n-s+\ell}. \qedhere\end{align*}
\end{proof}

The following is yet another way to write $\ehr_{\Pan_{r,s,n}}(t)$, which will be useful in our progress toward establishing Ehrhart positivity of panhandle matroids.
The proof requires a highly technical result, Lemma~\ref{lem:genfunc}, whose proof we defer to an appendix.

\begin{corollary}\label{formula for positivity}
The Ehrhart polynomial of $\Po_{\Pan_{r,s,n}}$ can be alternatively written as
\[\ehr_{\Pan_{r,s,n}}(t)=
\frac{n-s}{(n-1)!}\binom{t+n-s}{n-s}
\phi_{r,s,n}(t)\]
where
\begin{equation}\label{phisrlt}
\phi_{r,s,n}(t)
=\sum_{i=0}^{s-r}(-1)^i\binom{s}{i}\sum_{\ell=0}^{s-1}(n-2-\ell)!\ell!\binom{s-1-\ell-i+t(s-r-i+1)}{s-1-\ell}\binom{s-1-i+t(s-r-i)}{\ell}.
\end{equation}
\end{corollary}

\begin{proof}
Here, we assume that $t$ is a positive integer chosen to be large enough so that $s-1-(s-1)-(s-r)+t$ is nonnegative. This ensures that the upper term in each binomial coefficient, especially in those of the form $\binom{s-1-\ell-i+t(s-r-i+1)}{s-1-\ell}$, is nonnegative to avoid confusion between the two possible interpretations of binomial coefficients. 

Theorem~\ref{thm:betterPanhandleformula} implies
\begin{align*}
    \ehr_{{\Pan_{r,s,n}}}(t)&=\sum_{i=0}^{s-r}(-1)^i\binom{s}{i}\sum_{\ell=0}^{s-1}\binom{t(s-r-i)+s-1-i}{s-1-\ell}\binom{t+n-s}{n-s}\binom{t}{\ell}\frac{n-s}{n-s+\ell}\\
    &=(n-s)\binom{t+n-s}{n-s}\sum_{i=0}^{s-r}(-1)^i\binom{s}{i}\sum_{\ell=0}^{s-1}\binom{t(s-r-i)+s-1-i}{s-1-\ell}\binom{t}{\ell}\frac{1}{n-s+\ell}.\\
\end{align*}
It suffices to show that
\begin{align*}
&\sum_{\ell=0}^{s-1}\binom{t(s-r-i)+s-1-i}{s-1-\ell}\binom{t}{\ell}\frac{1}{n-s+\ell}\\
=&\frac{1}{(n-1)!}\sum_{\ell=0}^{s-1}(n-2-\ell)!\ell!\binom{s-1-\ell-i+t(s-r-i+1)}{s-1-\ell}\binom{s-1-i+t(s-r-i)}{\ell}.
\end{align*}
This follows from Lemma~\ref{lem:genfunc} with $u=t$.
\end{proof}

\begin{remark}\label{rmk:GeneralizeMinimal}
The matroid $\Pan_{r,r,n}$ is the \textit{minimal matroid} studied in~\cite{Ferroni2021A} (so called because it has the fewest bases among all connected matroids of rank~$r$ and ground set of size~$n$).
By Corollary \ref{formula for positivity}
\begin{align*}
\ehr_{{\Pan_{r,r,n}}}(t)&=\frac{n-r}{(n-1)!}\binom{t+n-r}{n-r}\sum_{\ell=0}^{r-1}(n-2-\ell)!\ell!\binom{r-1-\ell+t}{r-1-\ell}\binom{r-1}{\ell}\\
&= \frac{n-r}{(n-1)!}\binom{t+n-r}{n-r}\sum_{\ell=0}^{r-1}(n-r-1+\ell)!(r-1-\ell)!\binom{\ell+t}{\ell}\binom{r-1}{r-1-\ell}
\intertext{(replacing $\ell$ with $r-1-\ell$)}
&= \ehr_{{\Pan_{r,r,n}}}(t)=\frac{1}{\binom{n-1}{r-1}}\binom{t+n-r}{n-r}\sum_{\ell=0}^{r-1}\binom{n-r-1+\ell}{\ell}\binom{\ell+t}{\ell}
\end{align*}
by routine manipulation of factorials.
This formula is Theorem 1.6 in \cite{Ferroni2021A} for the Ehrhart polynomial of a minimal connected matroid.
\end{remark}

To conclude this subsection, we present the following conjecture which is supported by computational evidence. 

\begin{conjecture}
Consider the minimal matroid  $T_{r,n}$, the panhandle matroid $\Pan_{r,s,n}$, and the uniform matroid $U_{r,n}$.
The following inequality holds, coefficient-wise:
$$\ehr_{T_{r,n}}(t)\leq \ehr_{\Pan_{r,s,n}} \leq \ehr_{U_{r,n}}.$$
\end{conjecture}

\subsection{Ehrhart polynomials of stressed-hyperplane relaxations}

In \cite{Ferroni2021A}, Ferroni gives a formula for the Ehrhart polynomial $\ehr_{\widetilde M}(t)$ of the relaxation of a matroid $M$ by a circuit-hyperplane, in terms of $\ehr_M(t)$ and the Ehrhart polynomial of a minimal connected matroid.
We expand Ferroni's method to the case of a stressed-hyperplane relaxation, replacing the minimal connected matroid by a panhandle matroid.

\begin{lemma}\label{lem:relax-stressed}
Let $M$ be a rank $r$ matroid and let $H$ be a stressed hyperplane for $M$. 
Then each $r$-subset of $M$ with exactly one element outside $H$ is a basis of $M$.
\end{lemma}

\begin{proof}
Let $h_1,\dots,h_{r-1}$ be elements in $H$ and let $u$ be an element outside $H$. Since $H$ is a stressed hyperplane, rank$(\{h_1,\dots,h_{r-1}\})=r-1$.
Since $H$ is a hyperplane, we have that rank$(H\cup \{u\})=r$.
Therefore, there exists a basis $B\subset H\cup \{u\}$, and since $B$ cannot be a subset of $H$, $u\in B$.

Now, there exists an element in $b\in B\sm \{h_1,\dots,h_{r-1}\}$ such that $\{b,h_1,\dots,h_{r-1}\}$ is a basis of $M$.
Since $\{b,h_1,\dots,h_{r-1}\}$ is not a subset of $H$, we have that $b=u$. This completes the proof.
\end{proof}

\begin{theorem}\label{prop:Ehrhart-relaxation}
Let $M$ be a rank $r$ matroid on ground set~$[n]$ and let $H$ be a stressed hyperplane for $M$ with $|H|=s$.
Then the relaxation $\Rel_H(M)$ has Ehrhart function
\[
\ehr_{\Rel_H(M)}(t)=\ehr_{{M}}(t)+\sum_{m=0}^{t-1}\sum_{i=0}^{s-r}(-1)^i\binom{s}{i}\binom{t(s-r-i)+m+s-1-i}{s-1}\binom{m+n-s-1}{m}.
\]
\end{theorem}

\begin{proof}
Without loss of generality, assume that $H=[s]$.
Abbreviate $\widetilde{M}=\Rel_H(M)$.
Let $V$ be the set of vertices of $\Po_{M}$, and let $V'$ be the set of vertices in $\Po_{\widetilde{M}}$ corresponding to the new bases in $\binom{H}{r}$.
The set of vertices in $V$ that have an edge to a vertex in $V'$ corresponds to the bases of $M$ that have exactly one element not in $H$.
By Lemma \ref{lem:relax-stressed}, each $r$-subset with exactly one element not in $H$ is a basis of $M$.
So, the set of vertices in $V'$ together with the vertices of $V$ that are connected to a vertex in $V'$ make the vertices of the polytope $\Po_{\Pan_{r,s,n}}$.
Therefore, $\Po_{\widetilde{M}}=\Po_{\Pan_{r,s,n}}\cup \Po_{M}$.
The vertices of $\Po_{\Pan_{r,s,n}}\cap \Po_{M}$ correspond to the bases of $\Pan_{r,s,n}$ that have exactly one element outside $[s]$. 
(Alternatively, they are the indicator vectors of the bases for  $U_{r-1,s}\oplus U_{1,n-s}$, the same matroid appearing in Proposition \ref{prop:disconnected_stressed}.)
It follows from Proposition~\ref{Pan-characterization} that $\Po_{\Pan_{r,s,n}}\cap \Po_{M}$ is given by
\[
\left\{x\in [0,1]^n\ :\ \sum_{i=1}^nx_i=r \textrm{ and } \sum_{i=s+1}^n x_i = 1\right\}.
\]
As in the proof of Proposition~\ref{prop:Ehrhart-panhandle}, the Ehrhart function of this polytope is
\[\sum_{i=0}^{s-r}(-1)^i\binom{s}{i}\binom{t(s-r-i)+t+s-1-i}{s-1}\binom{t+n-s-1}{n-s-1}.\]
Now, by inclusion-exclusion,
\begin{align*}
\ehr_{\widetilde{M}}(t)&=\ehr_M(t) + \ehr_{{\Pan_{r,s,n}}}(t) - \sum_{i=0}^{s-r}(-1)^i\binom{s}{i}\binom{t(s-r-i)+t+s-1-i}{s-1}\binom{t+n-s-1}{n-s-1}\\
&=\ehr_M(t) + \sum_{m=0}^{t-1}\sum_{i=0}^{s-r}(-1)^i\binom{s}{i}\binom{t(s-r-i)+m+s-1-i}{s-1}\binom{m+n-s-1}{m}.
\end{align*}

The second equality comes from the following calculation
\begin{align*}
&\ehr_{{\Pan_{r,s,n}}}(t)-\sum_{i=0}^{s-r}(-1)^i\binom{s}{i}\binom{t(s-r-i)+t+s-1-i}{s-1}\binom{t+n-s-1}{n-s-1}\\
&=\sum_{m=0}^t\sum_{i=0}^{s-r}(-1)^i\binom{s}{i}\binom{t(s-r-i)+m+s-1-i}{s-1}\binom{m+n-s-1}{m}\\
&\hspace{2cm}-\sum_{i=0}^{s-r}(-1)^i\binom{s}{i}\binom{t(s-r-i)+t+s-1-i}{s-1}\binom{t+n-s-1}{n-s-1}\\
&=\sum_{m=0}^{t-1}\sum_{i=0}^{s-r}(-1)^i\binom{s}{i}\binom{t(s-r-i)+m+s-1-i}{s-1}\binom{m+n-s-1}{m}.\qedhere
\end{align*}
\end{proof}

As with Proposition~\ref{prop:Ehrhart-panhandle}, the previous formula is not entirely satisfactory because $t$ appears as a limit of summation. The following result gives a formula as a polynomial in $t$.

\begin{theorem} \label{thm:Ehrhart-relaxation-improved}
The Ehrhart polynomial for the relaxation $\Rel_H(M)$ can be alternatively written as
\begin{equation}
\ehr_{\Rel_H(M)}(t)=\ehr_{{M}}(t)+\frac{n-s}{(n-1)!}\binom{t-1+n-s}{n-s}
\tilde{\phi}_{r,s,n}(t)
\end{equation}
where
\begin{equation} \label{tildephi}
\tilde{\phi}_{r,s,n}(t)
=\sum_{i=0}^{s-r}(-1)^i\binom{s}{i}\sum_{\ell=0}^{s-1}(n-2-\ell)!\ell!\binom{s-2-\ell-i+t(s-r-i+1)}{s-1-\ell}\binom{s-1-i+t(s-r-i)}{\ell}.
\end{equation}
\end{theorem}


\begin{proof}
Here, we assume that $t$ is a positive integer chosen to be large enough so that $s-2-(s-1)-(s-r)+t$ is nonnegative. This ensures that the upper term in each binomial coefficient, especially in those of the form $\binom{s-2-\ell-i+t(s-r-i+1)}{s-1-\ell}$, is nonnegative to avoid confusion between the two possible interpretations of binomial coefficients.  

We rewrite the formula of Proposition~\ref{prop:Ehrhart-relaxation}, and then perform calculations similar to Corollaries~\ref{thm:betterPanhandleformula} and~\ref{formula for positivity}:
\begin{align*}
    \ehr_{\Rel_H(M)}(t)-\ehr_M(t)&=\sum_{m=0}^{t-1}\sum_{i=0}^{s-r}(-1)^i\binom{s}{i}\binom{t(s-r-i)+m+s-1-i}{s-1}\binom{m+n-s-1}{m}\\
    &=\sum_{m=0}^{t-1}\sum_{i=0}^{s-r}(-1)^i\binom{s}{i}\sum_{\ell=0}^{s-1}\binom{t(s-r-i)+s-1-i}{s-1-\ell}\binom{m}{m-\ell}\binom{m+n-s-1}{m}
\intertext{(using the Chu-Vandermonde identity \cite[p.~169]{concretemath})}
   &=\sum_{m=0}^{t-1}\sum_{i=0}^{s-r}(-1)^i\binom{s}{i}\sum_{\ell=0}^{s-1}\binom{t(s-r-i)+s-1-i}{s-1-\ell}\binom{n-s-1+\ell}{\ell}\binom{m+n-s-1}{m-\ell}
\intertext{(using the identity $\binom{K}{R}\binom{N}{K}=\binom{N-R}{K-R}\binom{N}{R}$)}
    &=\sum_{m=0}^{t-1}\sum_{i=0}^{s-r}(-1)^i\binom{s}{i}\sum_{\ell=0}^{s-1}\binom{t(s-r-i)+s-1-i}{s-1-\ell}\binom{n-s-1+\ell}{\ell}\binom{m+n-s-1}{n-s-1+\ell}\\
    &=\sum_{i=0}^{s-r}(-1)^i\binom{s}{i}\sum_{\ell=0}^{s-1}\binom{t(s-r-i)+s-1-i}{s-1-\ell}\binom{n-s-1+\ell}{\ell}\sum_{m=0}^{t-1}\binom{m+n-s-1}{n-s-1+\ell}\\
    &=\sum_{i=0}^{s-r}(-1)^i\binom{s}{i}\sum_{\ell=0}^{s-1}\binom{t(s-r-i)+s-1-i}{s-1-\ell}\binom{n-s-1+\ell}{\ell}\binom{t-1+n-s}{n-s+\ell}
\intertext{(using the ``hockey-stick identity'' $\sum_{I=R}^N \binom{I}{R} = \binom{N+1}{R+1}$ \cite[p.~160]{concretemath})}
    &=\sum_{i=0}^{s-r}(-1)^i\binom{s}{i}\sum_{\ell=0}^{s-1}\binom{t(s-r-i)+s-1-i}{s-1-\ell}\binom{t-1+n-s}{n-s}\binom{t-1}{\ell}\frac{n-s}{n-s+\ell}\\
    &=(n-s)\binom{t-1+n-s}{n-s}\sum_{i=0}^{s-r}(-1)^i\binom{s}{i}\sum_{\ell=0}^{s-1}\binom{t(s-r-i)+s-1-i}{s-1-\ell}\binom{t-1}{\ell}\frac{1}{n-s+\ell}.
\end{align*}
It suffices to show that
\begin{align*}
&\sum_{\ell=0}^{s-1}\binom{t(s-r-i)+s-1-i}{s-1-\ell}\binom{t-1}{\ell}\frac{1}{n-s+\ell}\\
=&\frac{1}{(n-1)!}\sum_{\ell=0}^{s-1}(n-2-\ell)!\ell!\binom{s-2-\ell-i+t(s-r-i+1)}{s-1-\ell}\binom{s-1-i+t(s-r-i)}{\ell}.
\end{align*}
This follows from Lemma~\ref{lem:genfunc} with $u=t-1$.
\end{proof}

\subsection{Ehrhart polynomials of paving matroids}
We now apply Proposition~\ref{prop:Ehrhart-relaxation} to paving matroids.
By~\cite[Prop.~3.9]{relaxation}, if $H_1,H_2$ are distinct stressed hyperplanes, then $H_2$ is a stressed hyperplane in $\Rel_{H_1}(M)$.
Hence by Proposition \ref{prop:stressed_hyps_paving}, every paving matroid can be relaxed to a uniform matroid, and we can iterate Proposition~\ref{prop:Ehrhart-relaxation}, we obtain the following formula.

\begin{proposition} \label{prop:ehrhart-paving}
Let $M$ be a rank $r$ paving matroid whose set of (stressed) hyperplanes is $\HH$, and for $r\leq s\leq n-1$, let $\HH_s=\{H\in\HH:\ |H|=s\}$.
Then
\[
\ehr_{M}(t)=\ehr_{U_{r,n}}(t)-\sum_{s=r}^{n-1}|\HH_s|\sum_{m=0}^{t-1}\sum_{i=0}^{s-r}(-1)^i\binom{s}{i}\binom{t(s-r-i)+m+s-1-i}{s-1}\binom{m+n-s-1}{m}.\qedhere
\]
\end{proposition}

\begin{theorem}\label{thm:Ehrhart-paving-improved-formula}
Let $M$ be a rank $r$ paving matroid whose set of (stressed) hyperplanes is $\HH$, and for $r\leq s\leq n-1$, let $\HH_s=\{H\in\HH:\ |H|=s\}$.  Then
\begin{align*}
&\ehr_{M}(t)=\ehr_{U_{r,n}}(t)-\sum_{s=r}^{n-1}|\HH_s|\frac{n-s}{(n-1)!}\binom{t-1+n-s}{n-s} \tilde\phi_{r,s,n}(t)
\end{align*}
where $\tilde\phi_{r,s,n}(t)$ is defined as in~\eqref{tildephi}.
\end{theorem}

\begin{remark} \label{rmk:sparse}
A matroid $M$ is \defterm{sparse paving} if it and its dual are both paving: equivalently, every subset of cardinality $r=\rank(M)$ is either a basis or a circuit-hyperplane.
In particular, $M$ is a paving matroid with no hyperplanes of size greater than~$r$.
Therefore, Theorem~\ref{thm:Ehrhart-paving-improved-formula} tells us that
\[
\ehr_{M}(t)=\ehr_{U_{r,n}}(t)-|\HH_r|\frac{1}{(n-1)!}\binom{t-1+n-r}{n-r}\sum_{\ell=0}^{r-1}(n-2-\ell)!\ell!\binom{t+r-2-\ell}{r-1-\ell}\binom{r-1}{\ell}.
\]
As in Remark \ref{rmk:GeneralizeMinimal}, we obtain
\begin{align*}
&\frac{1}{(n-1)!}\binom{t-1+n-r}{n-r}\sum_{\ell=0}^{r-1}(n-2-\ell)!\ell!\binom{t+r-2-\ell}{r-1-\ell}\binom{r-1}{\ell}\\
=&\frac{1}{\binom{n-1}{r-1}}\binom{t-1+n-r}{n-r}\sum_{\ell=0}^{r-1}\binom{n-r-1+\ell}{\ell}\binom{\ell+t-1}{\ell}\\
=&\ehr_{\Pan_{r,r,n}}(t-1).
\end{align*}
recovering~\cite[Corollary~4.6]{FerroniNotPositive}, i.e., the Ehrhart polynomial of a sparse paving matroid.
\end{remark}


\section{Progress toward Ehrhart positivity for panhandle matroids} \label{sec:ehrhart-panhandle}

In this section we describe our progress toward the following conjectures, both of which are supported by substantial computational evidence.

\begin{conjecture} \label{PanhandlePositive}
Panhandle matroids are Ehrhart positive.
\end{conjecture}

\begin{conjecture} \label{RelaxationPositive}
If $M$ is Ehrhart positive and $H$ is a stressed hyperplane in $M$, then $\Rel_H(M)$ is also Ehrhart positive.
\end{conjecture}

\begin{remark}
A matroid is a \defterm{positroid} \cite{KLS} if it can be represented by a real matrix whose maximal minors are all nonnegative.
Ferroni, Jochemko and Schr\"oter \cite{FerroniJochemkSchroter} conjectured that positroids are Ehrhart positive.
Our conjecture \ref{PanhandlePositive} is a special case of this conjecture, because lattice path matroids are positroids \cite[Lemma~23]{OhPositroids}.
\end{remark}

Recall the formula for $\ehr_{\Pan_{r,s,n}}(t)$ given in Corollary \ref{formula for positivity}.
Since $\binom{t+n-s}{n-s}$ is a polynomial in $t$ with positive coefficients, in order to prove Conjecture~\ref{PanhandlePositive}, it suffices to show that the polynomial $\phi_{r,s,n}(t)$ defined in~\eqref{phisrlt} has positive coefficients.
Analogously, in light of Theorem~\ref{thm:Ehrhart-relaxation-improved} and the observation that $\binom{t-1+n-s}{n-s}=t(t+1)\cdots(t-1+n-s)$ has nonnegative coefficients, we see that Conjecture~\ref{RelaxationPositive} will follow if the polynomial $\tilde\phi_{r,s,n}(t)$ defined in~\eqref{tildephi} can be shown to have positive coefficients.
Accordingly, we focus on the polynomials $\phi_{r,s,n}(t)$ and $\tilde\phi_{r,s,n}(t)$.

Interchanging the sums in the definitions of $\phi$ and $\tilde\phi$, we see that it is enough to show that
\begin{equation} \label{psisrlt}
\begin{aligned}
\psi_{s,r,\ell}(t) &=\sum_{i=0}^{s-r}(-1)^i\binom{s}{i}\binom{s-1-\ell-i+t(s-r-i+1)}{s-1-\ell}\binom{s-1-i+t(s-r-i)}{\ell},\\
\tilde{\psi}_{s,r,\ell}(t) &= \sum_{i=0}^{s-r}(-1)^i\binom{s}{i}\binom{s-2-\ell-i+t(s-r-i+1)}{s-1-\ell}\binom{s-1-i+t(s-r-i)}{\ell}
\end{aligned}
\end{equation}
have positive coefficients in $t$ for each $0\leq \ell\leq s-1$.


Note that the polynomial $\psi_{s,r,\ell}(t)$ is independent of $n$, implying the following: 

\begin{proposition}\label{prop:PanPosSmall s}
If $\psi_{s,r,\ell}(t)$ has positive coefficients for each $0\leq \ell\leq s-1$, then the polynomial $\ehr_{\Pan_{r,s,n}}(t)$ has positive coefficients for all $n\geq s+1$.
\end{proposition}

Therefore, the verification that finitely many polynomials have positive coefficients implies that an infinite number of panhandle matroids are Ehrhart positive.
We have verified using Sage that for all $1\leq r\leq s\leq 40$ and $0\leq \ell\leq s-1$, the polynomial $\psi_{s,r,\ell}(t)$ has positive coefficients. Thus, Proposition \ref{prop:PanPosSmall s} implies that for all $1\leq r\leq s\leq 40$, the panhandle matroid $\Pan_{r,s,n}$ is Ehrhart positive for all $n\geq s+1$.


We next describe our progress toward showing that $\psi_{s,r,\ell}(t)$ has positive coefficients.
Define
\[
\elm{n}{a,b}=\sum_{a\leq i_1<\dots<i_n\leq b} i_1\cdots i_n.
\]
(This quantity is notated $P^n_{a,b}$ rather than $\elm{n}{a,b}$ in \cite{FerroniHypersimplices}.)
The numbers $\elm{n}{a,b}$ generalize the unsigned Stirling numbers of the first kind $\stir{n}{k}$ (the number of permutations of $[n]$ with $k$ cycles).
They satisfy the following identities (among others), whch we will use freely:
\begin{align}
\elm{0}{a,b} &=1 &&\text{for all $a,b$}; \label{elm-one}\\
\elm{n}{a,b} &=0 &&\text{for $n<0$ or $n>b-a+1$};\label{elm-zero}\\
\elm{n}{1,b} &= \stir{b+1}{b+1-n};\label{elm-stirling}\\
\elm{n}{-a,b}&=\sum_{k=0}^n\elm{k}{-a,-1}\elm{n-k}{1,b}=\sum_{k=0}^n(-1)^k\elm{k}{1,a}\elm{n-k}{1,b} &&\text{for $-1<0<b$}\label{elm-sign}
\end{align}
Identity~\eqref{elm-stirling} follows from the standard generating function for first-Stirling numbers \cite[Proposition~1.3.7]{Stanley}, and~\eqref{elm-sign} results from classifying the summands in $\elm{n}{-a,b}$ by the number of negative factors.

\begin{proposition}\label{phicalc}
Let $\psi_{s,r,\ell}(t)$ and $\tilde\psi_{s,r,\ell}(t)$ be defined as in~\eqref{psisrlt}.
Then
\begin{align*}
&\psi_{s,r,\ell}(t)\\
&=\frac{1}{(s-1-\ell)!\ell!}\sum_{k=0}^{s-1}t^k\sum_{m=0}^k\sum_{i=0}^{s-r}(-1)^i\binom{s}{i}(s-r-i+1)^m(s-r-i)^{k-m} \elm{s-1-\ell-m}{-i+1,s-1-\ell-i} \elm{\ell-k+m}{s-\ell-i,s-1-i}
\intertext{and}
&\tilde{\psi}_{s,r,\ell}(t)\\
&=\frac{1}{(s-1-\ell)!\ell!}\sum_{k=0}^{s-1}t^k\sum_{m=0}^k\sum_{i=0}^{s-r}(-1)^i\binom{s}{i}(s-r-i+1)^m(s-r-i)^{k-m} \elm{s-1-\ell-m}{-i,s-2-\ell-i} \elm{\ell-k+m}{s-\ell-i,s-1-i}.
\end{align*}
\end{proposition}

\begin{proof}
We prove just the first assertion; the second proof is analogous.  The statement follows from the following calculation: 
\begin{align*}
&(s-1-\ell)!\ell!\psi_{s,r,\ell}(t)\\
&=\sum_{i=0}^{s-r}(-1)^i\binom{s}{i}\prod_{j=0}^{s-2-\ell}(t(s-r-i+1)+s-1-\ell-i-j)\prod_{j=0}^{\ell-1}(t(s-r-i)+s-1-i-j)\\
&=\sum_{i=0}^{s-r}(-1)^i\binom{s}{i}\left(\sum_{k=0}^{s-1-\ell}t^k(s-r-i+1)^k \elm{s-1-\ell-k}{-i+1,s-1-\ell-i}\right)\left(\sum_{k=0}^\ell t^k(s-r-i)^k \elm{\ell-k}{s-\ell-i,s-1-i}\right)\label{Psi:1}\\
&=\sum_{i=0}^{s-r}(-1)^i\binom{s}{i}\sum_{k=0}^{s-1}t^k\sum_{m=0}^k(s-r-i+1)^m \elm{s-1-\ell-m}{-i+1,s-1-\ell-i}(s-r-i)^{k-m} \elm{\ell-k+m}{s-\ell-i,s-1-i}\\
&=\sum_{k=0}^{s-1}t^k\sum_{m=0}^k\sum_{i=0}^{s-r}(-1)^i\binom{s}{i}(s-r-i+1)^m(s-r-i)^{k-m} \elm{s-1-\ell-m}{-i+1,s-1-\ell-i} \elm{\ell-k+m}{s-\ell-i,s-1-i}. \qedhere
\end{align*}
\end{proof}

\begin{corollary}
Let $s$ be a positive integer, and $0\leq r\leq s$, and $n\geq s-1$.

(1) Suppose that
\[
\zeta(r,s,k,\ell,m):=\sum_{i=0}^{s-r}(-1)^i\binom{s}{i}(s-r-i+1)^m(s-r-i)^{k-m} \elm{s-1-\ell-m}{-i+1,s-1-\ell-i} \elm{\ell-k+m}{s-\ell-i,s-1-i}
\]
is positive for all $0\leq \ell\leq s-1$, $0\leq k\leq s-1$, and $0\leq m\leq k$.
Then Conjecture~\ref{PanhandlePositive} holds.

(2) Suppose that
\[
\tilde\zeta(r,s,k,\ell,m):=\sum_{i=0}^{s-r}(-1)^i\binom{s}{i}(s-r-i+1)^m(s-r-i)^{k-m} \elm{s-1-\ell-m}{-i,s-2-\ell-i} \elm{\ell-k+m}{s-\ell-i,s-1-i}
\]
is positive for all $0\leq \ell\leq s-1$, $0\leq k\leq s-1$, and $0\leq m\leq k$.
Then Conjecture~\ref{RelaxationPositive} holds.
\end{corollary}

Therefore, we would like to show that $\zeta(r,s,k,\ell,m)$ is positive.
To that end, consider the power series
\begin{align*}
F_{s,k,\ell,m}(x)&=\sum_{i\geq 0}(-1)^i\binom{s}{i} \elm{s-1-\ell-m}{-i+1,s-1-\ell-i} \elm{\ell-k+m}{s-\ell-i,s-1-i}x^i,\\
\tilde F_{s,k,\ell,m}(x) &= \sum_{i\geq 0}(-1)^i\binom{s}{i} \elm{s-1-\ell-m}{-i,s-2-\ell-i} \elm{\ell-k+m}{s-\ell-i,s-1-i}x^i,\\
G_{s,k,\ell,m}(x)&=\sum_{n\geq 0}n^{k-m}(n+1)^mx^n,
\end{align*}
so that
\begin{subequations}
\begin{align}
F_{s,k,\ell,m}(x)G_{s,k,\ell,m}(x) &= \sum_r \zeta(r,s,k,\ell,m) x^{s-r}, \label{FGzeta} \\
\tilde F_{s,k,\ell,m}(x)G_{s,k,\ell,m}(x) &= \sum_r \tilde\zeta(r,s,k,\ell,m) x^{s-r}. \label{FGzetatilde}
\end{align}
\end{subequations}
First we consider $G_{s,k,\ell,m}(x)$.  As before, let $A(n,k)$ denote the Eulerian number which counts permutations of $[n]$ with exactly $k$ descents.
\textit{Worpitzky's identity} \cite[p.269]{concretemath} states that
\begin{equation} \label{worpitzky}
x^m = \sum_{a=0}^{m-1} A(m,a) \binom{x+a}{m}.
\end{equation}

We will also need the following formula for the product of binomial coefficients~\cite[p.171]{concretemath}:
\begin{equation}\label{binomials}
\binom{Q}{R} \binom{S}{T} = \sum_{U=0}^T \binom{R-Q+S}{U} \binom{T+Q-S}{T-U} \binom{Q+U}{R+T}.
\end{equation}

\begin{proposition}\label{prop:Gprop}
Let $s$ be a positive integer, $0\leq \ell\leq s-1$, $0\leq k\leq s-1$, and $0\leq m\leq k$.
Then $G_{s,k,\ell,m}(x) = P_{s,k,\ell,m}(x)/(1-x)^{k+1}$, where
\[
P_{s,k,\ell,m}(x)=\sum_{a=0}^{k-m-1}\sum_{b=0}^{m-1}\sum_{c=0}^m A(k-m,a)A(m,b)\binom{k-m+1+b-a}{c}\binom{m+a-b-1}{m-c}x^{k-a-c}.
\]
In particular, $P_{s,k,\ell,m}(x)$ is a polynomial with positive coefficients.
\end{proposition}

\begin{proof}
The statement follows from the following calculation
\begin{align*}
    &G_{s,k,\ell,m}(x)=\sum_{n\geq 0}n^{k-m}(n+1)^mx^n\\
    &=\sum_{a=0}^{k-m-1}A(k-m,a)\binom{n+a}{k-m}\sum_{b=0}^{m-1} A(m,b)\binom{n+1+b}{m}\sum_{n\geq0}x^n
    \intertext{(by~\eqref{worpitzky})}
    &=\sum_{a,b}A(k-m,a)A(m,b)\sum_{n\geq0}\sum_{c=0}^m\binom{k-m+1+b-a}{c}\binom{m+a-b-1}{m-c}\binom{n+a+c}{k}x^n
    \intertext{(by~\eqref{binomials}, with $Q = n+a$, $R=k-m$, $S = n+1+b$, $T = m$, $U = c$)}
    &=\sum_{a,b,c}A(k-m,a)A(m,b)\binom{k-m+1+b-a}{c}\binom{m+a-b-1}{m-c}\sum_{n\geq k-a-c}\binom{n+a+c}{k}x^n\\
    &=\sum_{a,b,c}A(k-m,a)A(m,b)\binom{k-m+1+b-a}{c}\binom{m+a-b-1}{m-c}\frac{x^{k-a-c}}{(1-x)^{k+1}}
\end{align*}
by the standard Taylor expansion of $1/(1-x)^{k+1}$.
Note that $a\leq k-m-1$ and $c\leq m$ implies $k-a-c>0$.
\end{proof}

We now return to $F_{s,k,\ell,m}$ and $\tilde F_{s,k,\ell,m}$.
Their coefficients are \textit{not} in general positive.
On the other hand, in light of Proposition~\ref{prop:Gprop}, we may rewrite~\eqref{FGzeta} and~\eqref{FGzetatilde} as
\begin{subequations}
\begin{align}
\sum_r \zeta(r,s,k,\ell,m) x^{s-r} &= \frac{F_{s,k,\ell,m}(x)}{(1-x)^{k+1}} \ P_{s,k,\ell,m}(x), \label{FGzeta:2}\\
\sum_r \tilde\zeta(r,s,k,\ell,m) x^{s-r} &= \frac{\tilde F_{s,k,\ell,m}(x)}{(1-x)^{k+1}} \ P_{s,k,\ell,m}(x). \label{FGzetatilde:2}
\end{align}
\end{subequations}

Since Proposition~\eqref{prop:Gprop} implies that $P_{s,k,\ell,m}$ has positive coefficients, Conjecture~\ref{PanhandlePositive} reduces to the problem of showing that $F_{s,k,\ell,m}(x)/(1-x)^{k+1}$ has positive coefficients.
A computation shows that the coefficient of $x^q$ in $F_{s,k,\ell,m}(x)\frac{1}{(1-x)^{k+1}}$ (resp., $\tilde F_{s,k,\ell,m}(x)\frac{1}{(1-x)^{k+1}}$) is the quantity $\xi(q,s,k,\ell,m)$ (resp., $\tilde\xi(q,s,k,\ell,m)$) defined by
\begin{align*}
\xi(q,s,k,\ell,m) &= \sum_{i=0}^q(-1)^i\binom{s}{i} \elm{s-1-\ell-m}{-i+1,s-1-\ell-i} \elm{\ell-k+m}{s-\ell-i,s-1-i} \binom{k+q-i}{k},\\
\tilde\xi(q,s,k,\ell,m) &= \sum_{i=0}^q(-1)^i\binom{s}{i} \elm{s-1-\ell-m}{-i,s-2-\ell-i} \elm{\ell-k+m}{s-\ell-i,s-1-i} \binom{k+q-i}{k}.
\end{align*}

We now describe a conjectured combinatorial interpretation for $\xi(q,s,k,\ell,m)$, which would imply Ehrhart positivity for panhandle matroids.
We suspect that there should be an analogous interpretation of $\tilde\xi(q,s,k,\ell,m)$, although we have not found one.
The main combinatorial objects are known as \textit{chain forests}, since they can be viewed as partially ordered sets that are disjoint unions of chains.

\begin{definition}
A \defterm{chain forest} is a partition of $[n]$ into an unordered set of blocks, each of which is internally ordered.
\end{definition}

We write chain forests with bars for delimiters between blocks: $S=B_1|\cdots|B_k$.
The number $k$ of blocks is the \defterm{length} of $S$, written $|S|$.
The first and last element of a block are its \defterm{leader} and its \defterm{trailer}.
The \defterm{standard representation} of a chain forest lists the blocks in increasing order by their leaders; e.g., $28|31|5|69|74$.
For example, the chain forests on $[3]$ are
\[\begin{array}{lll l lll l l}
123 & 213 & 312 && 12|3 & 1|23 & 13|2 && 1|2|3\\
132 & 231 & 321 && 21|3 & 1|32 & 2|31
\end{array}.\]

Let $S=B_1|\cdots|B_k$ be a chain forest, where $B_i=(b_{i_1},\dots,b_{i_{\ell_i}})$.
The \defterm{weight} of a block $B_i$ is the number of elements of $B_i$ that are less than its leader. 
Formally, notating the weight of $B_i$ by $\wt(B_i)$, we have
\[\wt(B_i) =  \#\{j\in[\ell_i]:\ b_{i_j}<b_{i_1}\}.\]
The weight of the chain forest $S$ is
\[
\wt(S)=\sum_{i=1}^k\wt(B_i).
\]
For example, $\wt(4\underline{3}5\underline{2}6)=2$ and $\wt(1|3\underline{2}|6\underline{45}|78)=3$.

If $S=B_1|\cdots|B_k$ is a standard representation, define
\[\gamma(S,\ell)=\#\{i\in[k]:\ |B_1|+\cdots+|B_i|>\ell\}=k-\max\{j:\ |B_1|+\cdots+|B_j|\leq\ell\}.\]

In other words, if we remove the delimiters from the standard representation of $S$, then $\gamma(S,\ell)$ is the number of trailers occurring after the $\ell$th position.
For example, the trailers of the chain forest $S=1|32|645|78$ are $1,2,5,8$.
After removing delimiters, we have $13264578$, with 1,2,5,8 in positions 1,3,6,8 respectively.  
So $\gamma(S,0)=|S|=4$, $\gamma(S,1)=\gamma(S,2)=3$, $\gamma(S,3)=2$, etc.

Define
\begin{align*}
\CF(n) &= \{\text{chain forests on $[n]$}\},\\
\CF(n,k) &= \{S\in\CF(n):\ |S|=k\},\\
\CF(q,n,k) &= \{S\in\CF(n,k):\ \wt(S)=q\},\text{ and}\\
\CF(q,n,k,\ell,m)&=\{S\in \CF(q,n,k):\ \gamma(S,\ell)=m\}.
\end{align*}
The \defterm{weighted Lah number} $W(q,n,k)$, studied by Ferroni~\cite{FerroniHypersimplices} is
\[W(q,n,k)=|\CF(q,n,k)|.\]
We will return to these numbers shortly, in Section~\ref{sec:weighted Lah}.

\begin{conjecture}\label{big conjecture}
Let
\[\bar\xi(q,s,k,\ell,m)=\xi(q,s,k-1,\ell,m-1)
= \sum_{i=0}^q(-1)^i\binom{s}{i} \elm{s-\ell-m}{-i+1,s-1-\ell-i} \elm{\ell-k+m}{s-\ell-i,s-1-i} \binom{k-1+q-i}{k-1}.\]
Then
\[
\bar\xi(q,s,k,\ell,m)=|\CF(q,s,k,\ell,m)|.
\]
\end{conjecture}

We have verified Conjecture~\ref{big conjecture}, using Sage \cite{Sage}, for all values of the parameters with $s\leq 7$.
We are therefore confident in its correctness (hence that of Conjecture~\ref{PanhandlePositive}), but have not yet found a general proof.


\subsection{A combinatorial proof of Ferroni's formula for weighted Lah numbers} \label{sec:weighted Lah}
Ferroni gave an algebraic proof of the following closed formula for the weighted Lah numbers:

\begin{theorem} \cite[Cor.~3.13]{FerroniHypersimplices} \label{weighted lah}
The weighted Lah numbers are given by the formula $W(q,n,k)=\eta(q,n,k)$, where
\begin{equation} \label{eqn:weighted lah}
\eta(q,n,k)=\sum_{j=0}^q\sum_{i=0}^{j} (-1)^{j+i} \binom{n}{j} \stir{j}{j-i}\stir{n-j}{k-j+i}\binom{k-1+q-j}{k-1}.
\end{equation}
\end{theorem}

An equivalent formula is
\begin{equation} \label{new weighted lah}
W(q,n,k)=\sum_{j=0}^q(-1)^j\binom{n}{j} \elm{n-k}{-j+1,n-1-j} \binom{k-1+q-j}{k-1},
\end{equation}
which can be transformed into Ferroni's original formula by applying~\eqref{elm-sign} followed by~\eqref{elm-stirling}.
There is a strong resemblance between $\eqref{new weighted lah}$ and the expression $\bar\xi(q,s,k,\ell,m)$ of Conjecture~\ref{big conjecture}.
Accordingly, we give a combinatorial proof of Theorem \ref{weighted lah} that we believe is of independent interest, and as mentioned before, that we hope will be useful in proving Conjecture~\ref{big conjecture}.

Let $\Sym_n$ denote the symmetric group of permutations of $[n]$.
We write $C(\sigma)$ for the set of cycles of a permutation $\sigma\in\Sym_n$, and $|c|$ for the length of a cycle $c$.

\begin{definition}
A \defterm{weighted permutation} of $[n]$ is a pair $(\sigma,x)$, where $\sigma\in\Sym_n$ and $x$ is a function $C(\sigma)\to\N$.
The \defterm{total weight} is $|x|=\sum_{c\in C(\sigma)}x(c)$.
A cycle $c\in C(\sigma)$ is \defterm{properly weighted} if $x(c)<|c|$; the pair $(\sigma,x)$ is \defterm{properly weighted} if all cycles are properly weighted.
\end{definition}

There is a straightforward bijection between
\begin{enumerate}
\item[(i)] properly weighted permutations $(\sigma,x)$ of $[n]$ with $k$ cycles and total weight $q$, and
\item[(ii)] chain forests $S$ on $[n]$ with $k$ blocks and weight $q$.
\end{enumerate}
Specifically, to construct the standard representation of $S$, write each cycle $c\in C(\sigma)$ with its $(x(c)+1)$th smallest element first, then sort the cycles by the first elements.
For example, if $\sigma=(1\;4\;8\;2\;6)(3)(5\;9\;7)$ is a permutation with cycles weighted 2,0,1 respectively, then the corresponding chain forest is $3|48261|759$ in standard form.

\begin{definition}
Let $Z\subseteq [n]$ and $b\geq 0$.
We define $S(n,k,q,Z,b)$ as the set of weighted permutations $(\sigma,x)$ such that:
\begin{itemize}
\item $\sigma\in\Sym_n$;
\item $|C(\sigma)|=k$;
\item $|x|=q$;
\item $b$ of the cycles consist only of elements in $Z$, and all such cycles are improperly weighted;
\item $k-b$ cycles consist only of elements in $[n]\sm Z$.
\end{itemize}
In particular, if $Z\neq\0$ then $S(n,k,q,Z,0)=\emptyset$, while $S(n,k,q,\0,0)$ is just the set of all weighted permutations of $[n]$ with total weight $q$ and $k$ cycles.
\end{definition}

\begin{lemma}\label{final lah form}
For all $n,k,q$ with $1\leq k\leq n$, we have
\[
\eta(q,n,k)=\sum_{j=0}^q\sum_{i=0}^j\sum_{Z\subseteq [n],\, |Z|=j}(-1)^{j-i}|S(n,k,q,Z,j-i)|.
\]
\end{lemma}

\begin{proof}
It suffices to show that
\begin{equation} \label{countSZ}
\stir{j}{j-i} \stir{n-j}{k-j+i}\binom{k-1+q-j}{k-1} = |S(n,k,q,Z,j-i)|,
\end{equation}
where $Z$ is any subset of $[n]$ of size $j$.
Indeed, let $(\sigma,x)\in S(n,k,q,Z,j-i)$.
Then $\sigma$ is described by a permutation of $Z$ with $j-i$ cycles $c_1,\dots,c_{j-i}$ ($\stir{j}{j-i}$ possibilities) together with a permutation of $[n]\sm Z$ with $k-j+i$ cycles $c_{j-i+1},\dots,c_k$ ($\stir{n-j}{k-j+i}$ possibilities).
Moreover, the weight function $x$ is given by a list of $k$ nonnegative integers with sum $q-|Z|=q-j$, namely $(x(c_1)-|c_1|,\,\dots,\,x(c_{j-i})-|c_{j-i}|,\,x(c_{j-i+1}),\,\dots,\,x(c_k))$, so by a standard stars-and-bars argument there are $\binom{k-1+q-j}{k-1}$ possibilities, establishing~\eqref{countSZ}.
\end{proof}

\begin{proof}[Proof of Theorem \ref{weighted lah}.]
We will employ the ``Garsia chi notation'': for a statement $P$, we define
\[\Gchi{P}=\begin{cases} 1 & \text{ if $P$ is true},\\ 0 & \text{ if $P$ is false.}\end{cases}\]

In light of Lemma~\ref{final lah form} and the bijection between weighted permutations and chain forests, it suffices to show that each weighted permutation $(\sigma,x)$ with $k$ blocks and total weight $q$ satisfies
\begin{equation} \label{Gchi-weighted-perm}
\sum_{i=0}^q\sum_{j=0}^{i}\sum_{Z\subseteq [n],\, |Z|=i}(-1)^{i-j}\Gchi{(\sigma,x)\in S(n,k,q,Z,i-j)}
~=~
\Gchi{(\sigma,x)\text{ is properly weighted}}.
\end{equation}

First, if $(\sigma,x)$ is properly weighted, then the $i=j=0$ summand is 1 and all other summands vanish, so~\eqref{Gchi-weighted-perm} holds.

Second, suppose that $(\sigma,x)$ is improperly weighted.
Let $\sigma=c_1\dots c_k$ be the cycle decomposition of $\sigma$ where, for convenience, $c_1,\dots,c_r$ are precisely the improperly weighted cycles.
Set $Z_B=\bigcup_{b\in B}c_b$ (here we are identifying $c_b$ with the set of its elements) for each $B\subseteq[r]$.
Then the sets $S(n,k,q,Z,i-j)$ containing $(\sigma,x)$ are precisely those for which $Z=Z_B$ for some $B\subseteq[r]$ with $|B|=i-j$.
Setting $m=i-j$, the left-hand side of~\eqref{Gchi-weighted-perm} becomes
\[\sum_{m=0}^r\sum_{B\subseteq[r],\: |B|=m}(-1)^m \Gchi{(\sigma,x)\in S(n,k,q,Z_B,m)}
=\sum_{m=0}^r(-1)^m\binom{r}{m}
= 0.\qedhere\]
\end{proof}


\section{Volumes of panhandle and paving matroid base polytopes}\label{sec:volumes}

In this section we obtain a volume formula for panhandle matroids, as an application of a more general formula due to Ashraf~\cite{Ashraf}.

\subsection{Ashraf's volume formula} \label{sec:Ashraf}
Throughout, we assume that $M$ is a connected, loopless matroid of rank~$r$ on ground set $[n]$.

A flat $F$ of $M$ is a \defterm{cyclic flat} if the restriction $ M|_F $ (i.e., the matroid on $F$ whose rank function is the restriction of that of $M$) has no coloops.
Equivalently, a cyclic flat is a flat that is a union of circuits.
We will be concerned with \defterm{anchored chains of cyclic flats:} sequences $\F=(F_0,\dots,F_k)$ of cyclic flats of $M$ such that
\[\0=F_0 \subsetneq F_1 \subsetneq \cdots \subsetneq F_k=[n].\]
The anchored chains of cyclic flats form a poset $\Gamma=\Gamma(M)$ under reverse containment: $ \F \leq \G $ if $\F\supseteq\G$.
Let $\hat\Gamma$ be the poset obtained from $\Gamma$ by adjoining a bottom element $ \hat{0} $.

\renewcommand{\aa}{\mathbf{a}}
\newcommand{\bb}{\mathbf{b}}
To each $\F=(F_0,\dots,F_k)\in\Gamma$, associate the binary string
\begin{multline*}
\bb_\F =
1^{\rank(F_1)} \
0^{|F_1|-\rank(F_1)} \
1^{\rank(F_2)-\rank(F_1)} \
0^{|F_2|-\rank(F_2)-|F_1|+\rank(F_1)} \\
\cdots \
1^{r-\rank(F_{k-1})}
0^{n-r-|F_{k-1}|+\rank(F_{k-1})}.
\end{multline*}
In particular, the prefix of $\bb_\F$ ending with the $i$th block of zeros has length $|F_i|$.

Given positive integers $r \leq n$, let $L(r,n)$ denote the set of all binary strings of length~$n$ that start with a one, end with a zero, and have exactly~$r$ ones.
We endow $ L(r,n) $ with the following partial order: for binary strings $ \aa = a_1 a_2 \cdots a_n $ and $ \bb = b_1 b_2 \cdots b_n $, we say $ \aa \leq \bb $ if for every positive integer $ j \leq n $,
\[
\sum_{i = 1}^{j} a_{i} \leq \sum_{i = 1}^{j} b_{i}.
\]

Let $w=w_1\cdots w_{n-1}\in\Sym_{n-1}$.  The \defterm{descent set} of $w$ is $\Des(w)=\{i\in[n-2]:\ w_i>w_{i+1}\}$.  The \defterm{descent string} of $w$ is the binary string  $ \bb_{\des} (w) = 1d_1d_2 \cdots d_{n-2} 0$, where
\[
	d_{i} = \begin{cases}
		1 &\text{if } w_i > w_{i+1},\\
		0 &\text{if } w_i < w_{i+1}.
	\end{cases}
\]
For $ \bb \in L(r,n) $,  define
\[
\delta_{\leq}(\bb) = \sum_{\aa\in L(r,n):\ \aa \leq \bb } \#\{w\in\Sym_{n-1} : \bb_{\des}(w) = \aa\}.
\]

Ashraf's volume formula \cite[Theorem~1.1]{Ashraf} states that for a connected matroid $M$ on~$[n]$ of rank~$r$, the normalized volume of its base polytope is given by
\begin{equation}
\label{AshrafThm}
	\Vol(\Po_{M}) = -\sum_{\F\in\Gamma(M)} \mu_{\hat\Gamma(M)}(\hat0,\F) \delta_{\leq}(\mathbf{b}_{\F}),
\end{equation}
where $\mu_{\hat\Gamma(M)}$ is the  M\"obius function for the poset $\hat\Gamma(M)$.

\subsection{Volumes of panhandle matroid base polytopes}

\begin{lemma}\label{lem:CycFlatsPan}
We have the following.
\begin{enumerate}
\item The panhandle matroid $\Pan_{r,s,n}$ has three cyclic flats, namely $\varnothing$, $[n]$, and $[s+1,n]$.
\item $\Pan_{r,s,n}$ has two anchored chains of cyclic flats, namely
$\F_0=(\0,[n])$ and $\F_1=(\0,[s+1,n],[n])$, and so
\[
\hat{\Gamma}(\Pan_{r,s,n}) =
\{ \hat{0} < \F_1 < \F_0 \}.
\]
\item $b_{\F_1} = 1 0^{n-s-1} 1^{r-1} 0^{s-r+1}$.
\end{enumerate}
\end{lemma}

\begin{proof}
(1) By Proposition \ref{prop:flats_for_panhandle}, the panhandle matroid $\Pan_{r,s,n}$ on $[n]$ has two types of flats:
\begin{itemize}
	\item Subsets of $ [s] $ of size at most $ r-1 $.
	\item Sets of the form $ [s+1,n] \cup A $, where $ A \subseteq [s] $ is a set of size at most $ r-2 $.
\end{itemize}
No nonempty flat $F$ of the first type is cyclic, because every element of $M|_F$ is a coloop.
On the other hand, if $F$ is a flat of the second type, then every element of $ A $ is a coloop in $M|_F$, so the only cyclic flat is $ [s+1, n] $ itself.

(2) These assertions are immediate from (1).

(3) The formula follows from the definition of $\bb_\F$ and the fact that $\rank([s+1,n]) = 1 $ and $\rank([n]) = r $ by Proposition \ref{prop:rank_panhandle}.
\end{proof}

Since $\hat\Gamma$ is a chain with three elements, we have $\mu_{\hat\Gamma}(\hat0,\F_0)=0$ and
$\mu_{\hat\Gamma}(\hat0,\F_1)=-1$, so the volume formula~\eqref{AshrafThm} simplifies to
\begin{align}
\Vol(\Po_{\Pan_{r,s,n}}) = \delta_{\leq}(\bb_{\F_1})
&=\delta_{\leq}(1\; 0^{n-s-1}\; 1^{r-1}\; 0^{s-r+1}) \notag\\
&= \#\left\{w\in\Sym_{n-1} : \bb_{\des}(w) \; \underset{L(r,n)}{\leq}\; 1\; 0^{n-s-1}\; 1^{r-1}\; 0^{s-r+1} \right\}\notag\\
&= \#\left\{w\in\Sym_{n-1} : \Des(w)\subseteq[n-s,n-2] \text{ and } |\Des(w)|=r-1\right\}\label{volume-from-Des}
\end{align}

To produce an explicit formula, define for $S=\{s_1<s_2<\cdots<s_{r-1}\}\subseteq[n-2]$
\begin{equation} \label{alpha-beta}
\begin{aligned}
\alpha_n(S) &= |\{w\in\Sym_{n-1} : \Des(w) \subseteq S\}|,\\
\beta_n(S) &= |\{w\in\Sym_{n-1} : \Des(w)=S\}|.
\end{aligned}
\end{equation}

By \cite[Proposition~1.4.1]{Stanley} we have
\begin{equation} \label{Equ:Alpha}
\alpha_n(S)=
\binom{n-1}{s_1,\ s_2-s_1,\ s_3-s_2,\ \dots,\ s_{r-1}-s_{r-2},\ n-1-s_{r-1}}
\end{equation}
and by inclusion/exclusion
\begin{equation} \label{Equ:Beta}
\beta_n(S) = \sum_{T\subseteq S} (-1)^{|S-T|} \alpha_n(T).
\end{equation}
By using~\eqref{Equ:Alpha} and~\eqref{Equ:Beta}, we obtain our final formula, which we state in a self-contained form.

\begin{theorem}\label{thm:volofpan}
The volume of the base polytope of the panhandle matroid $\Pan_{r,s,n}$ is
\begin{align*}
\Vol(\Po_{\Pan_{r,s,n}})
&= \sum_{\substack{S \subseteq [n-s,n-2] \\ |S| = r-1}} \beta_n(S)\\
&= \sum_{\substack{S \subseteq [n-s,n-2] \\ |S| = r-1}} \sum_{\substack{T\subseteq S\\ T=\{t_1<\cdots<t_k\}}} (-1)^{|S-T|} \binom{n-1}{t_1,\ t_2-t_1,\ \dots,\ t_k-t_{k-1},\ n-1-t_k}.
\end{align*}
\end{theorem}

\subsection{Volumes under relaxations of stressed hyperplanes} \label{volume-relax}

The formula for $\Vol(\Po_{\Pan_{r,s,n}})$ is key in describing how the volume of a matroid base polytope changes under relaxation of stressed hyperplanes. 
If $M=M_1\oplus M_2$ is disconnected, then $\Po_M=\Po_{M_1} \times \Po_{M_2}$ and so $\Vol(\Po_M) = \Vol(\Po_{M_1}) \Vol(\Po_{M_2})$.
By Proposition~\ref{prop:disconnected_stressed}, if $M$ has a stressed hyperplane, then both $\Po_{M_1}$ and $\Po_{M_2}$ are hypersimplices (in fact $\Po_{M_2}$ is a simplex), whose volume is given by~\eqref{eq:volume-hypersimplex}.
Thus, we restrict our discussion to connected matroids.

\begin{theorem}\label{thm:volume-relaxation}
Let $M$ be a connected rank-$r$ matroid on $[n]$ with a stressed hyperplane $H$ of cardinality $s$.
Then
\[\Vol(\Po_M) = \Vol(\Po_{\Rel_H(M)}) - \sum_{\substack{T \subseteq [n-s,n-2] \\ |T| = r-1}} \beta_n(T)\]
where $\beta_n(T)$ is as defined in~\eqref{Equ:Beta}.
\end{theorem}
\begin{proof}
As in the proof of Proposition~\ref{prop:Ehrhart-relaxation}, $\Po_{\Rel_H(M)} = \Po_{M} \cup \Po_{\Pan_{r,s,n}}$.
Moreover, $\Po_{M} \cap \Po_{\Pan_{r,s,n}}$ is a facet of $\Pan_{r,s,n}$.
We know that $\dim \Po_{\Rel_H(M)} =\dim\Po_M$ (since $M$ is connected), and therefore $\Vol(\Po_{M}) = \Vol(\Po_{\Rel_H(M)}) - \Vol(\Po_{\Pan_{r,s,n}})$.
The result follows by applying Theorem~\ref{thm:volofpan}.
\end{proof}

\begin{theorem} \label{thm:volume-paving}
Let $M$ be a connected rank-$r$ \underline{paving} matroid on $[n]$ whose set of hyperplanes is $\HH$.
For each $r \leq s \leq n$, let $\HH_s=\{H\in\HH:\ |H|=s\}$.
Then
\begin{align*}
\Vol(\Po_M)
&= A(n-1,r-1) - \sum_{s=r}^n |\HH_s| \sum_{\substack{T \subseteq [n-s,n-2] \\ |T| = r-1}} \beta_n(T)
\end{align*}
where $\beta_n(T)$ is as defined in~\eqref{Equ:Beta}.
\end{theorem}

\begin{proof}
The result follows from repeated applications of Theorem~\ref{thm:volume-relaxation}, using the volume formula for hypersimplices~\eqref{eq:volume-hypersimplex}.
\end{proof}

\section{Application: Steiner systems and projective planes} \label{subsec:steiner_systems}

Steiner systems are a class of combinatorial designs that include finite projective planes.
Every Steiner system gives rise to a paving matroid.
In this section, we specialize our results to give Ehrhart polynomials and volume formulas for matroid base polytopes corresponding to Steiner systems and finite projective planes.

We first briefly review the well-known theory of Steiner systems and matroids; see \cite[Section 6.1, Chapter 12]{Oxley} and \cite[Chapter 12]{Welsh}.

\begin{definition}
A \defterm{Steiner system} $S(t,k,n)$ consists of a \defterm{ground set} $E$ of $n$ points and a family~$\HH$ of $k$-subsets of $E$, called \defterm{blocks}, such that every $t$-subset of $E$ is contained in a unique block.
\end{definition}

Each block contains $\binom{k}{t}$ $t$-subsets of $E$, and no two blocks contain the same $t$-subset, so
\begin{equation*}\label{block_enum}
|\HH| = \frac{\binom{n}{t}}{\binom{k}{t}}.
\end{equation*}
There is in general no guarantee of existence or uniqueness of a Steiner system for particular parameters $t,k,n$.

\begin{definition}\label{def:PG}
Let $q\geq2$ be a positive integer.
A \defterm{projective plane of order~$q$} is a collection of $q^2+q+1$ points and $q^2+q+1$ lines, such that
\begin{itemize}
\item every line contains $q+1$ points;
\item every point lies on $q+1$ lines;
\item every two points lie on exactly one line;
\item every two lines intersect in exactly one point.
\end{itemize}
In particular, every projective plane of order~$q$ is a Steiner system $S(2,q+1,q^2+q+1)$, whose blocks are the lines.
\end{definition}

There is a standard construction of a projective plane of order~$q$ when $q$ is a prime power: the points and lines are respectively the 1- and 2-dimensional subspaces of a 3-dimensional vector space over the finite field with~$q$ elements.
Not every finite projective plane arises in this way; for example, there exist non-isomorphic projective planes of all prime powers $q\geq9$ \cite[p.702]{Codes}; see also \cite[\S6]{Cameron}.
(Projective geometries of dimensions greater than~2, do not in general give rise to Steiner systems or to paving matroids, so we do not consider them here.)

By \cite[Prop.~2.1.24]{Oxley}, a Steiner system $S(t,k,n)$ on ground set $E$ gives rise to a paving matroid on $E$ of rank $r=t+1$, whose hyperplanes are the blocks of $S$; in particular, all hyperplanes have cardinality~$k$.
See also \cite[p.202]{Welsh}.
In particular, a projective plane of order $q$ gives rise to a paving matroid of rank~3 whose hyperplanes all have cardinality $q+1$.
Accordingly, we can apply Theorem~\ref{thm:Ehrhart-paving-improved-formula} to write down the Ehrhart polynomials for these matroids, using the polynomial expression $\tilde\phi_{r,s,n}(t)$ defined in~\eqref{tildephi}.

\begin{proposition} \label{prop:ehrhart-Steiner}
Let $M$ be the rank-$r$ paving matroid corresponding to a Steiner system $S(r-1,k,n)$.
Then
\[
\ehr_{M}(t)=\ehr_{U_{r,n}}(t)-\frac{\binom{n}{r-1}}{\binom{k}{r-1}}\frac{n-k}{(n-1)!}\binom{t-1+n-k}{n-k} \tilde\phi_{r,k,n}(t).
\]
\end{proposition}

\begin{proposition} \label{prop:ehrhart-projective-plane}
Let $M$ be the rank-3 paving matroid corresponding to a finite projective plane of order~$q$, i.e., a Steiner system $S(2,q+1,q^2+q+1)$.
Then
\[
\ehr_{M}(t)=\ehr_{U_{3,n}}(t)-\frac{\binom{n}{2}}{\binom{q+1}{2}}\frac{q^2}{(q^2+q)!}\binom{t+q^2-1}{n-q-1} \tilde\phi_{3,q+1,q^2+q+1}(t).
\]
\end{proposition}

In particular, the base polytopes of non-isomorphic Steiner systems or projective planes with the same parameters have the same Ehrhart polynomial (and thus the same volume), even though the polytopes themselves are unlikely to be isomorphic.

\begin{example} \label{exa:Ehrhart-Fano}
Let $M$ be the paving (in fact, sparse paving) matroid corresponding to the Fano plane $PG(2,2)$.
Proposition~\ref{prop:ehrhart-projective-plane} and Katzman's formula~\eqref{eq:Katzman} give
\begin{align*}
\ehr_M(t)
&=\frac{1}{360}(t+1)(t+2)(116t^4+345t^3+553t^2+486t+180).
\end{align*}
\end{example}

We now specialize the volume formula of Theorem~\ref{thm:volume-paving} to Steiner systems and projective planes gives the following.
Recall the definition of $\beta_n(T)$ in~\eqref{alpha-beta} and~\eqref{Equ:Beta}.

\begin{proposition} \label{prop:Steiner-volume}
\begin{enumerate}
\item The normalized volume of the base polytope of the matroid of a Steiner system $S(r-1,k,n)$ is
\[
A(n-1,r-1) - \frac{\binom{n}{r-1}}{\binom{k}{r-1}} \sum_{\substack{T \subseteq [n-k,n-2] \\ |T| = r-1}} \beta_n(T).
\]
\item In particular, the normalized volume of the base polytope of the matroid of a projective plane of order $q$ is
\[
A(q^2+q,2) - (q^2+q+1) \sum_{\substack{T \subseteq [q^2,q^2+q-1] \\ |T| = 2}} \beta_{q^2+q+1}(T).
\]
\end{enumerate}
\end{proposition}

To evaluate this sum in practice, it is helpful to note that $A(n,2)=3^n - (n+1)2^n + n(n+1)/2$ (sequence A000460 in \cite{OEIS}).

\begin{example} \label{volume-Fan}
Let $M$ be the paving matroid corresponding to the Fano plane $PG(2,2)$.  Here $q=2$ and $[q^2,q^2+q-1]=\{4,5\}$ itself has size~2, so Proposition~\ref{prop:Steiner-volume}~(2) becomes
\begin{align*}
\Vol(\mathcal{P}_M) &= A(6,2) - 7\beta_7(\{4,5\})\\
&= A(6,2)-7\Big|\big\{w\in\Sym_6:\ \Des(w)=\{4,5\}\big\}\Big|\\
&= 302 - 7\cdot10 = 232.
\end{align*}
Indeed, this result is consistent with the Ehrhart polynomial calculated in Example~\ref{exa:Ehrhart-Fano}, whose leading term is $116/360t^6 = 232t^6/6!$.
\end{example}


\section*{Acknowledgements}

The authors thank Mohsen Aliabadi, Margaret Bayer, Matthias Beck, Federico Castillo, Luis Ferroni, Joseph Kung, and James Oxley for fruitful correspondence.
This work was initiated at the 2021 Graduate Research Workshop in Combinatorics.


\appendix

\section{An identity of generating functions}

The following technical lemma is necessary for Corollary~\ref{formula for positivity}.

\begin{lemma}\label{lem:genfunc}
Let $r\leq s < n$ be positive integers. Let $u$ be a nonnegative integer and $t$ an indeterminate.

Then
\begin{align*}
&\sum_{\ell=0}^{s-1}\binom{t(s-r-i)+s-1-i}{s-1-\ell}\binom{u}{\ell}\frac{1}{n-s+\ell}\\
=&\frac{1}{(n-1)!}\sum_{\ell=0}^{s-1}(n-2-\ell)!\ell!\binom{t(s-r-i)+u+s-1-\ell-i}{s-1-\ell}\binom{t(s-r-i)+s-1-i}{\ell}.
\end{align*}
\end{lemma}
\begin{proof}
First, it is elementary that
\[\int_0^xy^{n-s-1}(1+y)^u\mathrm{d}y=\sum_{\ell=0}^u\binom{u}{\ell}\frac{1}{n-s+\ell}x^{n-s+\ell}\]
so the coefficient of $x^{n-1}$ in the power series expansion of $(1+x)^{t(s-r-i)+s-1-i}\int_0^xy^{n-s-1}(1+y)^u\mathrm{d}y$ is
\[\sum_{\ell=0}^{s-1}\binom{t(s-r-i)+s-1-i}{s-1-\ell}\binom{u}{\ell}\frac{1}{n-s+\ell}.\]

For a function $f$, let $D^kf$ denote the $k$th derivative of $f$ with respect to $x$. Because for a power series $F(x)$, the coefficient of $x^{n-1}$ is given by $\frac{1}{(n-1)!}D^{n-1}F(0)$, the statement of the lemma is equivalent to the following:
\begin{align*}
&D^{n-1}\Big((1+x)^{t(s-r-i)+s-1-i}\int_0^xy^{n-s-1}(1+y)^u\mathrm{d}y\Big)\Big\rvert_{x=0}\\
=&\sum_{\ell=0}^{s-1}(n-2-\ell)!\ell!\binom{t(s-r-i)+u+s-1-\ell-i}{s-1-\ell}\binom{t(s-r-i)+s-1-i}{\ell}.
\end{align*}
\begin{claim}\label{derivative claim}
For $k\geq 0$, the following holds
\begin{align*}
    &D^{k}\Big((1+x)^{t(s-r-i)+s-1-i}\int_0^xy^{n-s-1}(1+y)^u\mathrm{d}y\Big)\\
    =&\sum_{\ell=0}^{k-1}D^{k-1-\ell}\Big(x^{n-s-1}(1+x)^{t(s-r-i)+u+s-1-i-\ell}\prod_{j=0}^{\ell-1}(t(s-r-i)+s-1-i-j)\Big)\\
    &+(1+x)^{t(s-r-i)+s-1-i-k}\prod_{j=0}^{k-1}(t(s-r-i)+s-1-i-j)\int_0^xy^{n-s-1}(1+y)^u\mathrm{d}y.
\end{align*}
\end{claim}
\begin{proof}[Proof of Claim \ref{derivative claim}.]
We proceed by induction. The claim is true for $k=0$. Let $k\geq 1$, and assume the claim holds for $k-1$. Then
\begin{align*}
    &D^{k}\Big((1+x)^{t(s-r-i)+s-1-i}\int_0^xy^{n-s-1}(1+y)^u\mathrm{d}y\Big)\\
    =&\sum_{\ell=0}^{k-2}D^{k-1-\ell}\Big(x^{n-s-1}(1+x)^{t(s-r-i)+u+s-1-i-\ell}\prod_{j=0}^{\ell-1}(t(s-r-i)+s-1-i-j)\Big)\\
    &\quad+D\Big((1+x)^{t(s-r-i)+s-1-i-(k-1)}\prod_{j=0}^{k-2}(t(s-r-i)+s-1-i-j)\int_0^xy^{n-s-1}(1+y)^u\mathrm{d}y\Big)\\
    =&\sum_{\ell=0}^{k-2}D^{k-1-\ell}\Big(x^{n-s-1}(1+x)^{t(s-r-i)+u+s-1-i-\ell}\prod_{j=0}^{\ell-1}(t(s-r-i)+s-1-i-j)\Big)\\
    &\quad+(1+x)^{t(s-r-i)+s-1-i-(k-1)}\prod_{j=0}^{k-2}(t(s-r-i)+s-1-i-j)x^{n-s-1}(1+x)^u\mathrm{d}y\\
    &\quad+(1+x)^{t(s-r-i)+s-1-i-k}\prod_{j=0}^{k-1}(t(s-r-i)+s-1-i-j)\int_0^xy^{n-s-1}(1+y)^u\mathrm{d}y\\
    =&\sum_{\ell=0}^{k-1}D^{k-1-\ell}\Big(x^{n-s-1}(1+x)^{t(s-r-i)+u+s-1-i-\ell}\prod_{j=0}^{\ell-1}(t(s-r-i)+s-1-i-j)\Big)\\
    &\quad+(1+x)^{t(s-r-i)+s-1-i-k}\prod_{j=0}^{k-1}(t(s-r-i)+s-1-i-j)\int_0^xy^{n-s-1}(1+y)^u\mathrm{d}y.
\qedhere\end{align*}
\end{proof}

Using Claim \ref{derivative claim},
\begin{align*}
    &D^{n-1}\Big((1+x)^{t(s-r-i)+s-1-i}\int_0^xy^{n-s-1}(1+y)^u\mathrm{d}y\Big)\Big\rvert_{x=0}\\
    =&\sum_{\ell=0}^{n-2}D^{n-2-\ell}\Big(x^{n-s-1}(1+x)^{t(s-r-i)+u+s-1-i-\ell}\prod_{j=0}^{\ell-1}(t(s-r-i)+s-1-i-j)\Big)\Big\rvert_{x=0}\\
    &\quad+(1+x)^{t(s-r-i)+s-1-i-(n-1)}\prod_{j=0}^{n-2}(t(s-r-i)+s-1-i-j)\int_0^xy^{n-s-1}(1+y)^u\mathrm{d}y\Big\rvert_{x=0}\\
    =&\sum_{\ell=0}^{n-2}D^{n-2-\ell}\Big(x^{n-s-1}(1+x)^{t(s-r-i)+u+s-1-i-\ell}\prod_{j=0}^{\ell-1}(t(s-r-i)+s-1-i-j)\Big)\Big\rvert_{x=0}\\
    =&\sum_{\ell=0}^{n-2}\left[\prod_{j=0}^{\ell-1}(t(s-r-i)+s-1-i-j)\right]\\
    &\quad\sum_{p=0}^{n-2-\ell}\binom{n-2-\ell}{p}D^p\big(x^{n-s-1}\big)D^{n-2-\ell-p}\big((1+x)^{t(s-r-i)+u+s-1-i-\ell}\big)\Big\rvert_{x=0}\\
    =&\sum_{\ell=0}^{s-1}\left[\prod_{j=0}^{\ell-1}(t(s-r-i)+s-1-i-j)\right]\\
    &\quad\sum_{p=0}^{n-2-\ell}\binom{n-2-\ell}{p}D^p\big(x^{n-s-1}\big)D^{n-2-\ell-p}\big((1+x)^{t(s-r-i)+u+s-1-i-\ell}\big)\Big\rvert_{x=0}\\
    =&\sum_{\ell=0}^{s-1}\left[\prod_{j=0}^{\ell-1}(t(s-r-i)+s-1-i-j)\right]\\
    &\quad\binom{n-2-\ell}{n-s-1}(n-s-1)!\prod_{j=0}^{s-2-\ell}(t(s-r-i)+u+s-1-i-\ell-j)\\
    =&\sum_{\ell=0}^{s-1}(n-2-\ell)!\ell!\binom{t(s-r-i)+s-1-i}{\ell}\binom{t(s-r-i)+u+s-1-i-\ell}{s-1-\ell}
\end{align*}
where the fourth and fifth equalities use the fact that
$D^p\big(x^{n-s-1}\big)\Big\rvert_{x=0}=0$
unless $p=n-s-1$.
\end{proof}

\bibliographystyle{amsplain}
\bibliography{bibliography}

\end{document}